\documentclass[12pt]{amsart}
\usepackage{amsthm}
\usepackage{lscape}
\usepackage{amssymb}
\usepackage{cite}
\usepackage{geometry}
\usepackage{enumerate}
\usepackage{setspace}
\usepackage{tikz}
\usepackage{tikz-cd}
\usetikzlibrary{matrix, arrows}
\usepackage[unicode=true]
 {hyperref}
\hypersetup{
colorlinks=true,
urlcolor=black,
citecolor=blue,
linkcolor=blue,
}

\allowdisplaybreaks

\newtheorem{thm}{Theorem}[section]
\newtheorem{prop}[thm]{Proposition}
\newtheorem{lem}[thm]{Lemma}
\newtheorem{cor}[thm]{Corollary}

\theoremstyle{definition}

\newtheorem*{assumption}{Assumption}

\theoremstyle{remark}
\newtheorem{remark}[thm]{Remark}

\numberwithin{equation}{section}

 \newcommand{\op}{\mbox{\tiny $\mathrm{op}$}}
  \newcommand{\Aut}{\mathrm{Aut}}

\begin{document}

\large 

\title{A generalization of Ito's theorem to skew braces}
\author{Cindy (Sin Yi) Tsang}
\address{Department of Mathematics, Ochanomizu University, 2-1-1 Otsuka, Bunkyo-ku, Tokyo, Japan}
\email{tsang.sin.yi@ocha.ac.jp}
\urladdr{http://sites.google.com/site/cindysinyitsang/} 

\date{\today}

\maketitle

%\vspace{-5mm}

\begin{abstract}The famous theorem of It\^{o} in group theory states that if a group $G=HK$ is the product of two abelian subgroups $H$ and $K$, then $G$ is metabelian. We shall generalize this to the setting of a skew brace $(A,{\cdot\,},\circ)$. Our main result says that if $A = BC$ or $A = B\circ C$ is the product of two trivial sub-skew braces $B$ and $C$ which are both left and right ideals in the opposite skew brace of $A$, then $A$ is meta-trivial. One can recover It\^{o}'s Theorem by taking $A$ to be an almost trivial skew brace.
\end{abstract}

%%\vspace{-1mm}

\setcounter{tocdepth}{1}
\tableofcontents

\vspace{-6mm}
 
\section{Introduction}

A \emph{skew brace} is a set $(A,{\cdot \,},\circ)$ equipped with two group operations $\cdot$ and $\circ$ such that the so-called brace relation 
\[ a\circ (b\cdot c) = (a\circ b)\cdot a^{-1}\cdot (a\circ c)\]
holds for all $a,b,c\in A$. The groups $(A,{\cdot\, })$ and $(A,\circ)$ have the same identity element which we denote by $1$. For $a\in A$, we shall write $a^{-1}$ for its inverse in $(A,{\cdot\, })$ and $\overline{a}$ for its inverse in $(A,\circ)$. For $a,b\in A$, let us also define
\[ a * b = a^{-1}(a\circ b) b^{-1},\]
and note that $a*b = 1$ exactly when $a\circ b= ab$. Thus, this binary operation $*$ measures the difference of $\cdot$ and $\circ$. We say that $(A,{\cdot \,},\circ)$ is \emph{trivial} if
\[ a\circ b =ab \mbox{ or equivalently }a*b =1\mbox{ for all }a,b\in A,\]
and $(A,{\cdot\,},\circ)$ is \emph{almost trivial} if
\[ a\circ b =ba \mbox{ or equivalently }a*b =a^{-1}bab^{-1}\mbox{ for all }a,b\in A.\]
Identifying groups with almost trivial skew braces, we may therefore view $*$ as a generalization of the commutator operation, and trivial skew brace as a generalization of abelian group. 

%It is well-known that
%\[ \lambda_a : (A,{\cdot\,})\longrightarrow (A,{\cdot \,});\,\ \lambda_a(b) = a^{-1}(a\circ b)\]
%is a group automorphism and that
%\[ \lambda : (A,\circ) \longrightarrow \Aut(A,{\cdot\, });\,\ \lambda(a) = \lambda_a\]
%is a group homomorphism. We refer the reader to \cite{Skew braces} for details. Let us just remark that skew left brace was introduced as tool to study non-degenerate set-theoretic solutions, but we shall not explore this connection here.

\vspace{2mm}

The notion of skew brace was first introduced in \cite{Skew braces} as a tool to study the Yang-Baxter equation. We shall not explore this connection here so we refer the reader to \cite{Skew braces} for the background.

\vspace{2mm}

%For any elements $a,b$ in a group $(A,{\cdot\,})$, we shall write
%\[ a^b = bab^{-1}\mbox{ and }[a,b] = abab^{-1} = aa^b.\]
%A simple calculation shows that
%\begin{equation}
%\label{commutator}
%[a,bc]  = [a,b] [a,c]^b \mbox{ and }[ab,c] = [b,c]^a[a,c]
%\end{equation}
%for all $a,b,c\in A$. 

The purpose of this paper is to generalize the following results from group theory in the setting of skew braces. %Proposition \ref{prop:class2} is an easy exercise so we shall not .% Proposition \ref{prop:metabelian} is due It\^{o} \cite{Ito} and may be proven via a simple commutator calculation; also see \cite[Theorem 3.1.6]{product book}. 

\begin{prop}\label{prop:class2}Let $G=HK$ be a group which is a product of two abelian normal subgroups $H$ and $K$. Then $[[G,G],G] =1$, namely $G$ is nilpotent of class at most two.
\end{prop}
\begin{proof}
Since $G=HK$ with both $H$ and $K$ abelian, we have
\begin{equation}\label{commutator}
[G,G] = [H,H][K,K][H,K] = [H,K].
\end{equation}
Notice that $[H,K]\subseteq H\cap K$ because $H$ and $K$ are normal subgroups. Since $H$ and $K$ are abelian, elements of $[G,G]$ commute with those of $H$ and $K$. But $G=HK$, so this means that $[G,G]$ is contained in the center of $G$.
\end{proof}

\begin{prop}\label{prop:metabelian}Let $G=HK$ be a group which is a product of two abelian subgroups $H$ and $K$. Then $[G,G]$ is abelian, namely $G$ is metabelian.
\end{prop}
\begin{proof}This is a result due to It\^{o} \cite{Ito} and may be proven using a very simple commutator calculation; also see \cite[Theorem 3.1.6]{product book}. 
\end{proof}

Let $(A,{\cdot \,},\circ)$ be a skew brace. A subset $B$ of $A$ is a \emph{sub-skew brace} if $B$ is a subgroup of both $(A,{\cdot\, })$ and $(A,\circ)$. In this case, clearly $(B,{\cdot \,},\circ)$ is a skew brace. A subset $I$ of $A$ is an \emph{ideal} if $I$ is a normal subgroup of both $(A,{\cdot\, })$ and $(A,\circ)$ such that $a\cdot I = a\circ I$ for all $a\in A$. In this case, the quotient
\[ (A/I,{\cdot\, },\circ),\mbox{ where } A/I = (A,{\cdot\,})/(I,{\cdot \,})= (A,\circ)/(I,\circ),\]
endowed with the induced operations $\cdot$ and $\circ$ is a skew brace. 
\vspace{2mm}

For any subsets $X$ and $Y$ of $A$, define $X*Y$ to be the subgroup of $(A,{\cdot\, })$ generated by the elements $x*y$ with $x\in X$ and $y\in Y$. A subgroup $I$ of $(A,{\cdot \,})$ is a \emph{left ideal} if $A*I\subseteq I$, and \emph{right ideal} if $I*A\subseteq I$. When $(A,{\cdot\, },\circ)$ is an almost trivial skew brace identified with the group $(A,{\cdot\, })$, we have
\begin{align*}
 \mbox{sub-skew brace} &= \mbox{subgroup},\\
 \mbox{ideal} & =\mbox{normal subgroup},\\
\mbox{left ideal} &=\mbox{normal subgroup},\\
 \mbox{right ideal} &=\mbox{normal subgroup}.
 \end{align*}
But in general, ideal, left ideal, and right ideal are different notions. Put
\[ A' = A*A  ,\,\ A^3 = A*A',\,\ A^{(3)} = A'*A.\]
We shall say that $(A,{\cdot\, },\circ)$ is \emph{meta-trivial} if $A'$ is a trivial skew brace. This terminology comes from the fact that $A'$ is the smallest ideal of $A$ for which the quotient skew brace is trivial. We also say that $(A,{\cdot\, },\circ)$ is \emph{left nilpotent of index at most $3$} if $A^{3} =1$, and similarly \emph{right nilpotent of index at most $3$} if $A^{(3)} = 1$. When $(A,{\cdot\, },\circ)$ is an almost trivial skew brace identified with the group $(A,{\cdot\, })$, observe that
\[ A' = [A,A],\,\ A^{3} =[A,A'] = [A',A] = A^{(3)},\]
and so in particular we have
\begin{align*}
\mbox{meta-trivial}&=\mbox{metabelian},\\
\mbox{left nilpotent of index at most $3$} &=\mbox{nilpotent of class at most $2$},\\
\mbox{right nilpotent of index at most $3$} &=\mbox{nilpotent of class at most $2$}.
\end{align*}
But in general, left nilpotency and right nilpotency are different notions. We refer the reader to \cite{nilpotent} for more details on nilpotency of skew braces.

\vspace{2mm}

First, Proposition \ref{prop:class2} may be generalized in two ways as follows.

\begin{thm}\label{thm:left} Let $(A,{\cdot \,},\circ)$ be a skew brace such that $A = B\circ C$ for some sub-skew braces $B$ and $C$ satisfying all of the following:
\begin{enumerate}[$(1)$]
\item $B$ and $C$ are trivial skew braces,
\item $B$ and $C$ are normal subgroups of $(A,\circ)$,
\item $B$ and $C$ are right ideals in $A$.
\end{enumerate} 
Then $A^3=1$, namely $A$ is left nilpotent of index at most $3$.
\end{thm}

\begin{thm}\label{thm:right} Let $(A,{\cdot \,},\circ)$ be a skew brace such that $A = B\cdot C$ for some sub-skew braces $B$ and $C$ satisfying all of the following:
\begin{enumerate}[$(1)$]
\item $B$ and $C$ are trivial skew braces,
\item $B$ and $C$ are normal subgroups of $(A,{\cdot\, })$,
\item $B$ and $C$ are left ideals in $A$.
\end{enumerate} 
Then $A^{(3)}=1$, namely $A$ is right nilpotent of index at most $3$.
\end{thm}

We remark that Theorem \ref{thm:right} was already known by \cite[Theorem 3.5]{factorization}. But we shall still give a proof which we believe is a bit easier than the one in \cite{factorization}.

\vspace{2mm}

It is clear that Proposition \ref{prop:class2} is a special case of both Theorems \ref{thm:left} and \ref{thm:right} when $(A,{\cdot\, },\circ)$ is taken to be an almost trivial skew brace.

\vspace{2mm}

Next, we consider It\^{o}'s Theorem (Proposition \ref{prop:metabelian}) whose generalization is much more difficult. As noted in \cite[Example 3.7]{factorization}, there exists a skew brace $(A,{\cdot\, },\circ)$ such that $A=B\cdot C$ for some sub-skew braces $B$ and $C$ satisfying: 
\begin{enumerate}[$(1)$]
\item $B$ and $C$ are left ideals in $A$ which are trivial as skew braces,
\item $B$ is a normal subgroup of $(A,{\cdot\, })$ but $C$ is not,
\end{enumerate} 
but $(A,{\cdot\, },\circ)$ is not meta-trivial. This means that we cannot simply drop the normality condition. The condition of being left or right ideals in $A$ have to be removed also because they correspond to normal subgroups in an almost trivial skew brace.

\vspace{2mm}

Our idea is to use the opposite skew brace introduced in \cite{opposite}. The \emph{opposite} \par \noindent of a skew brace $(A,{\cdot \,},\circ)$ is the same set $A$ equipped with the operations
\[ a \cdot'b = b\cdot a \mbox{ and }a\circ' b = a\circ b.\]
In other words, we replace the dot operation with its opposite $\cdot^{\op}$ and keep the circle operation $\circ$ unchanged. One easily checks that $(A,{\cdot^{\op}\, },\circ)$ is also a skew brace, and we shall denote it by $A^{\op}$. Here is our main result:

\begin{thm}\label{thm:ito}Let $(A,{\cdot \,},\circ)$ be a skew brace such that $A = B \cdot C$ or $A=B\circ C$ for some sub-skew braces $B$ and $C$ satisfying both of the following:
\begin{enumerate}[$(1)$]
\item $B$ and $C$ are trivial skew braces,
\item $B$ and $C$ are both left and right ideals in $A^{\op}$.
\end{enumerate} 
Then $A'$ is a trivial skew brace, namely $A$ is meta-trivial.
\end{thm}

It\^{o}'s Theorem may be recovered from Theorem \ref{thm:ito} when $(A,{\cdot\, },\circ)$ is taken to be an almost trivial skew brace (the opposite skew brace is trivial and so condition (2) is redundant in this case). 

\begin{remark}\label{remark intro}In Theorems \ref{thm:left}, \ref{thm:right}, and \ref{thm:ito}, although only one of $A = B\cdot C$  and $A = B \circ C$ is assumed, as we shall show in Lemma \ref{lem:product}, we in fact have both factorizations by the other hypotheses.\end{remark}

In Section \ref{construct sec}, we shall describe a method to construct skew braces $(A,{\cdot\,},\circ)$ which admit exact factorizations $A = B\cdot C$ and $A=B \circ C$ by some sub-skew braces $B$ and $C$ satisfying all of the following:
\begin{enumerate}[(i)]
\item $B$ and $C$ are trivial skew braces,
\item $B$ is a left ideal in $A^{\op}$ and $C$ is a right ideal in $A^{\op}$,
\item $B$ and $C$ are left ideals in $A$.
\end{enumerate}
In Section \ref{ex sec}, we then apply our construction to exhibit a family of such skew braces $(A,{\cdot\, },\circ)$ which are not meta-trivial. This means that condition (2) in Theorem \ref{thm:ito} cannot be relaxed to condition (ii), even when (iii) is imposed.

\begin{remark}
In the theory of group factorizations, we also have the Kegel-Wielandt theorem. But as illustrated in \cite[Example 2.15]{factorization}, one cannot expect a naive analogue of this for skew braces. Perhaps our new approach to bring the opposite skew brace into the picture can shed light on this problem.\end{remark}

\section{Basic facts about skew braces}

In this section, let $A=(A,{\cdot \,},\circ)$ be a skew brace, and let $A^{\op}=(A,{\cdot^{\op}\,},\circ)$ denote its opposite skew brace.

\vspace{2mm}

It is well-known that every skew brace gives rise to a \emph{lambda map}. For the skew brace $A$, the associated lambda map is the group homomorphism
\[ \lambda: (A,\circ) \longrightarrow \mathrm{Aut}(A,{\cdot\,});\,\ a\mapsto\lambda_a,\mbox{ where } \lambda_a(b) = a^{-1}(a\circ b).\]
In terms of this map $\lambda$, we then have
\[ a\circ b = a\lambda_a(b),\,\ ab = a\circ \lambda_{\overline{a}}(b),\,\ a*b = \lambda_a(b)b^{-1}.\]
It follows from the definition that
\[ (A,{\cdot \,},\circ)\mbox{ is a trivial skew brace} \iff \lambda_a(b) = b\mbox{ for all }a,b\in A. \]
For any subgroup $I$ of $(A,{\cdot\,})$, we clearly have
\begin{align*}
I\mbox{ is a left ideal in $A$}&\iff \lambda_a(x) \in I\mbox{ for all }x\in I,\, a\in A,\\
I\mbox{ is a right ideal in $A$}&\iff \lambda_x(a)a^{-1} \in I\mbox{ for all }x\in I,\, a\in A.
\end{align*}
Notice that $\mathrm{Aut}(A,{\cdot\,})=\mathrm{Aut}(A,{\cdot^{\op}})$. For the skew brace $A^{\op}$, converting the operation $\cdot^{\op}$ to ${\cdot\, }$, the associated lambda map is the group homomorphism
\[ \lambda^{\op}: (A,\circ) \longrightarrow \mathrm{Aut}(A,{\cdot\,});\,\ a\mapsto\lambda^{\op}_a,\mbox{ where } \lambda_a^{\op}(b) = (a\circ b)a^{-1}.\]
In terms of this map $\lambda^{\op}$, we then have
\[ a\circ b = \lambda^{\op}_a(b)a,\,\ ab = b\circ \lambda^{\op}_{\overline{b}}(a),\,\ a*b = a^{-1}\lambda_a^{\op}(b)ab^{-1}.\]
It follows from the definition that
\[ (A,{\cdot \,},\circ)\mbox{ is a trivial skew brace} \iff \lambda^{\op}_a(b) = aba^{-1}\mbox{ for all }a,b\in A. \]
For any subgroup $I$ of $(A,{\cdot\,})$, as above we have
\begin{align*}
I\mbox{ is a left ideal in $A^{\op}$}&\iff \lambda^{\op}_a(x) \in I\mbox{ for all }x\in I,\, a\in A,\\
I\mbox{ is a right ideal in $A^{\op}$}&\iff a^{-1}\lambda^{\op}_x(a) \in I\mbox{ for all }x\in I,\, a\in A.
\end{align*}
Note that subgroups of $(A,{\cdot\, })$ and $(A,{\cdot^{\op}})$ are the same thing.

\vspace{2mm}

When $(A,{\cdot\, },\circ)$ is an almost trivial skew brace, we have
\[\lambda_a(b) = a^{-1}ba \mbox{ and }\lambda_a^{\op}(b) = b\]
 for all $a,b\in A$. In other words, the action of $\lambda_a$ coincides with conjugation by $a^{-1}$, while $\lambda_a^{\op}$ gives the trivial action.
 
 \vspace{2mm}
 
The commutator $[a,b] = aba^{-1}b^{-1}$ in groups satisfies the identities
\begin{align*}\notag
[a,bc]  &= [a,b] \cdot b [a,c]  b^{-1},\\\label{commutator}
[ab,c] & = a [b,c]  a^{-1}\cdot [a,c].
\end{align*} 
They may be generalized in the context of skew braces as follows. 

\begin{lem}\label{formulas1}
For any $a,b,c\in A$, we have the identities
\begin{align*}
a*(bc) &= (a*b)\cdot b  \cdot(a*c) \cdot b^{-1},\\ 
(a\circ b)* c &= (a *(b*c)) \cdot (b*c)\cdot (a*c).
\end{align*}
In particular, this implies that
\begin{align*}
 a * (bc) &= a*b &&\hspace{-3.5cm}\mbox{whenever }a*c=1,\\
 (a\circ b)*c & =a*c &&\hspace{-3.5cm}\mbox{whenever }b*c=1.
 \end{align*}
\end{lem}
\begin{proof}This is \cite[Lemma 2.1]{factorization}, but for the convenience of the reader, we give a proof here as well. We compute that
\begin{align*}
 a*(bc) &= \lambda_a(bc)(bc)^{-1}\\
 & = \lambda_a(b)\lambda_a(c) c^{-1}b^{-1}\\
& = \lambda_a(b)b^{-1}\cdot b\cdot \lambda_a(c) c^{-1}\cdot b^{-1}\\
&  =(a*b)\cdot b \cdot(a*c)  \cdot b^{-1},\\
 (a\circ b)* c& = \lambda_{a\circ b}(c)c^{-1}\\
 & = \lambda_a(\lambda_b(c))c^{-1}\\
 & = \lambda_a(\lambda_b(c)c^{-1})\lambda_a(c)c^{-1}\\
 & = \lambda_a(b*c)(b*c)^{-1}(b*c) (a*c)\\
 & = (a *(b*c)) \cdot (b*c)\cdot (a*c),
 \end{align*}
which is as claimed.
\end{proof}

Rewriting the first identity in Lemma \ref{formulas1}, we obtain:

\begin{cor}\label{cor:conj}For any $a,b,c\in A$, we have the identity
\[ a(b*c)a^{-1} = (b*a)^{-1} (b*(ac)).\]
\end{cor}

The commutator $[a,b] = aba^{-1}b^{-1}$ in groups behaves very nicely under the conjugation action $b^a = aba^{-1}$, namely $[b,c]^a= [b^a,c^a]$. The next lemma is a generalization of this in the context of skew braces.

\begin{lem}\label{formulas2}For any $a,b,c\in A$, we have the identity
\[ \lambda_a(b*c)  = (a\circ b \circ \overline{a}) * \lambda_a(c).\]
%\begin{align*}
%a(b*c)a^{-1} & =  (b*a)^{-1}\cdot (b*ac),\\
%\lambda_a(b*c) &  = (a\circ b \circ \overline{a}) * \lambda_a(c).
%\end{align*}
\end{lem}
\begin{proof} We compute that
\begin{align*}
%a(b*c)a^{-1} & = a\lambda_b(c)c^{-1}a^{-1}\\
%& =  a\lambda_b(a)^{-1} \cdot \lambda_b(ac)(ac)^{-1}\\
%& = (b*a)^{-1}\cdot (b*ac),
\lambda_a(b*c) &= \lambda_a(\lambda_b(c)c^{-1})\\
&= \lambda_{a\circ b}(c) \lambda_{a}(c)^{-1}\\
&= \lambda_{a\circ b\circ\overline{a}}(\lambda_a(c)) \lambda_{a}(c)^{-1}\\
&= (a\circ b\circ \overline{a}) * \lambda_a(c),
\end{align*}
 which is as claimed.
\end{proof}

To prove Theorems \ref{thm:left}, \ref{thm:right}, and \ref{thm:ito}, we shall need to show that
\[ A*A' =1 ,\,\ A' * A =1,\,\ A'*A' =1, \]
respectively. For any subgroups $B$ and $C$ of $(A,{\cdot \,})$, obviously
\[ B*C = 1\iff b*c =1\mbox{ for all }b\in B\mbox{ and }c\in C.\]
Below, we show that it suffices to consider generators of $B$ and $C$, provided that the generators satisfy suitable conditions.

\begin{prop}\label{prop:gen} Let $B$ and $C$ be subgroups of $(A,{\cdot \,})$ which have $X$ and $Y$ as generating sets, respectively. 
\begin{enumerate}[$(a)$]
\item If $b*y= 1$ for all $b\in B$ and $y\in Y$, then $B*C = 1$.
\item If $X\subseteq Y$, and $x*y=1$ for all $x\in X$ and $y\in Y$, then $B*C=1$.
%\end{enumerate}
%Provided that one of the conditions
%\begin{enumerate}[$(i)$]
%\item $\lambda_{\overline{x}}(X)\subseteq X$ for all $x\in X$
%\item $\lambda_{\overline{x}}^{\op}(X)\subseteq X$ for all $x\in X$
%\end{enumerate}
%is satisfied, the following holds as well.
%\begin{enumerate}[$(a)$]\setcounter{enumi}{+1}
%\item If $x *c =1$ for all $x\in X$ and $c\in C$, then $B*C=1$.
%\item If $x*y = 1$ for all $x\in X$ and $y\in Y$, then $B*C=1$.
\end{enumerate}
\end{prop}

%We shall in fact apply (b) and (c) only under condition (i), but we include condition (ii) here as well for the sake of completeness.

\begin{proof}[Proof of $(a)$] Note that $b*y^{-1}=1$ holds for all $b\in B$ and $y\in Y$ because 
\[b*y=1\iff \lambda_b(y) = y \iff \lambda_b(y^{-1}) =y^{-1}\iff b*y^{-1}=1. \]
Now, let $b\in B$ and $c\in C$. We may write
\[ c = y_1^{\epsilon_1}y_2^{\epsilon_2}\cdots y_\ell^{\epsilon_\ell}\mbox{ for some }y_1,y_2,\dots,y_\ell\in Y,\, \epsilon_1,\epsilon_2,\dots,\epsilon_\ell\in \{\pm1\}.\]
For $\ell=1$, we clearly have $b*c = b*y_1^{\epsilon_1}=1$. For $\ell\geq 2$, we have
\begin{align*}
 b*c &= b* (y_1^{\epsilon_1}(y_2^{\epsilon_2}\cdots y_\ell^{\epsilon_\ell}))\\\
 & = (b*y_1^{\epsilon_1})\cdot y_1^{\epsilon_1}\cdot (b*(y_2^{\epsilon_2}\cdots y_{\ell}^{\epsilon_\ell}))\cdot y_1^{-\epsilon_1} \end{align*}
by Lemma \ref{formulas1}. It then follows by induction on $\ell$ that $b*c=1$.
\end{proof}

\begin{proof}[Proof of $(b)$] Note that $x*y^{-1}=1$ holds for all $x\in X$ and $y\in Y$ as above. Moreover, observe that
\begin{equation}\label{xy}
 x* y^{\epsilon} = 1\iff \lambda_x(y^\epsilon) = y^{\epsilon} \iff y^\epsilon = \lambda_{\overline{x}}(y^{\epsilon}). \end{equation}
Since $X\subseteq Y$, this in particular implies that
\[ \overline{x} = \lambda_{\overline{x}}(x^{-1}) = x^{-1}\mbox{ for all }x\in X.\]
This in turn yields that
\[ x^{-1} * y = \lambda_{\overline{x}}(y)y^{-1} = yy^{-1}=1\mbox{ for all }x\in X\mbox{ and }y\in Y.\]
Thus, replacing $b$ by $x$ and $x^{-1}$, the same proof as in (a) shows that
\[ x*c = 1= x^{-1}*c\mbox{ for all }x\in X\mbox{ and }c\in C.\]
Now, let $b\in B$ and $c\in C$. We may write
\[ b = x_1^{\epsilon_1}x_2^{\epsilon_2}\cdots x_\ell^{\epsilon_\ell}\mbox{ for some }x_1,x_2,\dots,x_\ell\in X,\, \epsilon_1,\epsilon_2,\dots,\epsilon_\ell\in \{\pm1\}.\]
For $\ell=1$, we clearly have $b*c = x_1^{\epsilon_1}*c=1$. For $\ell\geq 2$, we have
\begin{align}\label{step}
b * c & = (x_1^{\epsilon_1}(x_2^{\epsilon_2}\cdots x_\ell^{\epsilon_\ell}))*c\\\notag
&= (x_1^{\epsilon_1}\circ \lambda_{\overline{x_1^{\epsilon_1}}}(x_2^{\epsilon_2}\cdots x_\ell^{\epsilon_\ell})) *c\\\notag
&= (x_1^{\epsilon_1}\circ (\lambda_{\overline{x_1^{\epsilon_1}}}(x_2^{\epsilon_2})\cdots \lambda_{\overline{x_1^{\epsilon_1}}}(x_\ell^{\epsilon_\ell}))) *c.
\end{align}
Since $\overline{x_1} = x_1^{-1}$, the subscript $\overline{x_1^{\epsilon_1}}$ is either $x_1$ or $\overline{x_1}$. In either case, because $X\subseteq Y$, we deduce from (\ref{xy}) and Lemma \ref{formulas1} that
\begin{align*}
b * c &= (x_1^{\epsilon_1}\circ (x_2^{\epsilon_2}\cdots x_\ell^{\epsilon_\ell})) *c\\
& = (x_1^{\epsilon_1} * ((x_2^{\epsilon_2}\cdots x_\ell^{\epsilon_\ell})*c)) ((x_2^{\epsilon_2}\cdots x_\ell^{\epsilon_\ell})*c)(x_1^{\epsilon_1}*c).
\end{align*}
It then follows by induction on $\ell$ that $b*c=1$.
\end{proof}

%\begin{proof}[Proof of $(c)$] For each fixed $x\in X$, we have $x*c=1$ for all $c\in C$ by the same proof as in (a). The claim now follows from (b).
%\end{proof} 

\begin{remark}In step (\ref{step}) of Proposition \ref{prop:gen}(b), we converted ${\cdot\,}$ to $\circ$ because the operation $*$ does not behave well with respect to ${\cdot\, }$ in the first argument. This is why we had to impose the extra condition $X\subseteq Y$ and the discussion prior to (\ref{step}) is needed. But this can be avoided in the case of a two-sided skew brace.

\vspace{2mm}

Recall that $(A,{\cdot\, },\circ)$ is said to be a \emph{two-sided skew brace} if 
\[ (a\cdot b)\circ c = (a\circ c)\cdot c^{-1} \cdot (b\circ c)\]
also holds for all $a,b,c\in A$. In this case, we have
\begin{align*}
(ab)*c & = (ab)^{-1}((ab)\circ c)c^{-1}\\
& = b^{-1}a^{-1} (a\circ c)c^{-1}(b\circ c)c^{-1}\\
& = b^{-1}\cdot a^{-1}(a\circ c)c^{-1} \cdot b\cdot b^{-1}(b\circ c)c^{-1}\\
& = b^{-1}\cdot (a*c)\cdot b \cdot (b*c),
\end{align*}
analogous to the first identity in the Lemma \ref{formulas1}. We can then use the same argument as in (a) to prove (b) without having to assume that $X\subseteq Y$. We do not need this fact but we mention it here for completeness.
\end{remark}

%We shall only apply Proposition \ref{prop:gen}(b),(c) when $X \subseteq Y$. In this case, the conditions (i),(ii) can actually be dropped.

%As an immediate consequence, we obtain:

\begin{cor}\label{cor:gen} Let $S$ be a generating set of $A'$ as a subgroup of $(A,\cdot)$.
\begin{enumerate}[$(a)$]
\item If $a*y=1$ for all $a\in A$ and $y\in S$, then $A*A'=1$.
\item If $x*a=1$ for all $x\in S$ and $a\in A$, then $A'*A=1$.
\item If $x*y=1$ for all $x,y\in S$, then $A'*A'=1$.
\end{enumerate}
\end{cor}
\begin{proof}In (a),(b), and (c), respectively, take 
\[ (B,C,Y) = (A,A',S),\,\ (B,C,X,Y) = (A',A,S,A),\, (A',A',S,S),\]
and then apply Proposition \ref{prop:gen}. Observe that our choices of $X,Y$ in (b),(c) clearly satisfy $X\subseteq Y$ and so Proposition \ref{prop:gen}(b) indeed applies.
\end{proof}

We end this section with the following simple fact that was mentioned in Remark \ref{remark intro} in the introduction.
 
 \begin{lem}\label{lem:product}Let $B$ and $C$ be any sub-skew braces of $A$.
 \begin{enumerate}[$(a)$]
\item Assume that $A=B\cdot C$. If $B$ is a left ideal in $A$ or $A^{\op}$, then $A=B\circ C$.
\item Assume that $A=B\circ C$. If $B$ is a right ideal in $A$ or $A^{\op}$, then $A=B\cdot C$.
\end{enumerate}
\end{lem}
\begin{proof} For any $b\in B$ and $c\in C$, observe that
\begin{align*}
 cb &= c\circ \lambda_{\overline{c}}(b),&&\mbox{$\lambda_{\overline{c}}(b)\in B$ if $B$ is a left ideal in $A$},\\
bc& = c\circ \lambda_{\overline{c}}^{\op}(b), &&\mbox{$\lambda_{\overline{c}}^{\op}(b)\in B$ if $B$ is a left ideal in $A^{\op}$},\\
b\circ c&= b\lambda_{b}(c)c^{-1}\cdot c,&& \mbox{$\lambda_{b}(c)c^{-1}\in B$ if $B$ is a right ideal in $A$},\\
b\circ c &=  c\cdot c^{-1}\lambda_{b}^{\op}(c)b, &&\mbox{$c^{-1}\lambda_{b}^{\op}(c)\in B$ if $B$ is a right ideal in $A^{\op}$}.
\end{align*}
We then deduce the claims since $A=B\cdot C$ is equivalent to $A=C\cdot B$ and similarly $A=B\circ C$ is equivalent to $A=C \circ B$. \end{proof}

 \section{Products of trivial skew braces}

In this section, let $A=(A,{\cdot \,},\circ)$ be a skew brace. Also let $B$ and $C$ be any sub-skew braces of $A$. We shall establish some facts that we need in order to prove Theorems \ref{thm:left} and \ref{thm:right}. They are essentially \cite[Lemma 3.4]{factorization}, but both of $B$ and $C$ are assumed to be trivial skew braces and left ideals there. We shall prove them using as few hypotheses as possible.

\vspace{2mm}

Recall that $B*C$ is defined to be the subgroup of $(A,{\cdot \, })$ generated by the elements $b*c$ with $b\in B$ and $c\in C$. 

\begin{prop}\label{lem:normal}Assume that $A = B\cdot C$. If $B$ is a trivial skew brace, then $B*C$ is a normal subgroup of $(A,{\cdot \, })$. 
\end{prop}
\begin{proof}Let $a\in A,\, b\in B,\,  c\in C$. Since $A = BC$, we may write 
\[ a = yx\mbox{ and }ac = vu\mbox{ with }x,u\in B\mbox{ and } y,v\in C.\]
Note that $b* x = b*u = 1$ because $B$ is a trivial skew brace. It then follows from Corollary \ref{cor:conj} and Lemma \ref{formulas1} that
\begin{align*}
a(b*c)a^{-1} & = (b*a)^{-1}(b*(ac))\\
 &=(b*(yx))^{-1}(b*(vu))\\
& = (b*y)^{-1}  (b*v),
\end{align*} 
which lies in $B*C$. Thus, indeed $B*C$ is a normal subgroup of $(A,{\cdot \, })$.  
\end{proof}

\begin{prop}\label{lem:invariant}Assume that $A = B\cdot C$. If one of the conditions
\begin{enumerate}[$(a)$]
\item $C$ is a trivial skew brace and is a left ideal in $A$
\item $B$ is a trivial skew brace and is normal subgroup of $(A,\circ)$
 \end{enumerate}
holds, then $B*C$ is a left ideal in $A$.
%the generating set
%\[ \{ b *c : b\in B,\, c\in C\}\]
%of $B*C$ is invariant under $\lambda_a$ for any $a\in A$, so in particular $B*C$ is also a left ideal in $A$.
\end{prop}

We do not need this proposition to prove our theorems, but we include it here for the sake of completeness.

\begin{proof}[Proof of $(a)$] Let $a\in A,\, b\in B,\,  c\in C$. We may write
\[ a \circ b\circ \overline{a} = x\circ y \mbox{ with } x\in B\mbox{ and }y\in C\]
because $A=B\circ C$ by Lemma \ref{lem:product}. Since $C$ is a trivial skew brace and is a left ideal in $A$, we also know that
\[ \lambda_a(c)\in C\mbox{ and }y * \lambda_a(c) = 1.\]
It then follows from Lemmas \ref{formulas1} and \ref{formulas2} that
\begin{align*}
\lambda_a(b*c) & = (a\circ b \circ \overline{a}) * \lambda_a(c)\\
& = (x\circ y) *  \lambda_a(c)\\
& = x * \lambda_a(c),
\end{align*}
which lies in $B*C$. Thus, indeed $B*C$ is a left ideal in $A$.
\end{proof}
\begin{proof}[Proof of $(b)$] Let $a\in A,\, b\in B,\,  c\in C$. We may write
\[ \lambda_a(c) = yx \mbox{ with } x\in B\mbox{ and }y\in C\]
because $A=B\cdot  C$. Since $B$ is a trivial skew brace and is a normal subgroup \par\noindent of $(A,\circ)$, we also know that
\[ a\circ b\circ \overline{a}\in B\mbox{ and }(a\circ b\circ \overline{a}) * x =1.\]
It then follows from Lemmas \ref{formulas1} and \ref{formulas2} that
\begin{align*}
\lambda_a(b*c) & = (a\circ b \circ \overline{a}) * \lambda_a(c)\\
& = (a\circ b \circ \overline{a}) * (yx)\\
& = (a\circ b \circ \overline{a})* y,
\end{align*}
which lies in $B*C$. Thus, indeed $B*C$ is a left ideal in $A$.
\end{proof}

Th next proposition may be viewed as a generalization of (\ref{commutator}).

\begin{prop}\label{prop:A'}Assume that $A=B\cdot C$ and $A=B\circ C$. If both $B$ and $C$ are trivial skew braces, then $A' = (B*C)\cdot (C*B)$.
% and so in particular
%\[ \{ b*c : b\in B,\, c\in C\} \cup \{ c*b : c\in C,\, b\in B\}\]
%is a generating set of $A'$ as a subgroup of $(A,{\cdot \, })$.
\end{prop}
\begin{proof}Recall that $A' = A*A$ is the subgroup of $(A,{\cdot \, })$ generated by $a_1*a_2$ with $a_1,a_2\in A$. Since $A = B\cdot  C$, we may write
\[a_2= b_2c_2\mbox{ with }b_2\in B\mbox{ and }c_2\in C.\]
Using Lemma \ref{formulas1}, we compute that
\begin{align*}
a_1 * a_2 & = a_1 * (b_2c_2)\\
& = (a_1*b_2) \cdot b_2\cdot (a_1*c_2)\cdot b_2^{-1}.
\end{align*}
Since $A = B\circ C$, we may similarly write
\[ a_1 = c_1\circ b_1 = b_3\circ c_3\mbox{ with }b_1,b_3\in B\mbox{ and }c_1,c_3\in C.\]
We have $b_1*b_2=1 = c_3*c_2$ because both $B$ and $C$ are trivial skew braces. Again by Lemma \ref{formulas1}, we get that
\begin{align*}
a_1*b_2 & = (c_1\circ b_1) * b_2 = c_1 * b_2,\\
a_1 * c_2 & = (b_3\circ c_3) *c_2 = b_3 *c_2.
\end{align*}
Thus, we have shown that
\[ a_1 * a_2 =  (c_1 *b_2)\cdot b_2\cdot (b_3 * c_2)\cdot b_2^{-1},\]
which lies in $(C*B)\cdot (B*C)$ because  $B*C$ is a normal subgroup of $(A,{\cdot\,})$ by Proposition \ref{lem:normal}. This proves $\subseteq$, and the inclusion $\supseteq$ is trivial. \end{proof}

We can now prove Theorems \ref{thm:left} and \ref{thm:right}. %More tools are needed in order to prove Theorem \ref{thm:ito}, and we shall resume our discussion of products of trivial skew braces in Section \ref{part II}.

\section{Proof of Theorem \ref{thm:left}}

Let $(A,{\cdot \,},\circ)$ be a skew brace such that $A = B\circ C$ for sub-skew braces $B$ and $C$ satisfying conditions (1) -- (3) stated in the theorem. By Lemma \ref{lem:product}, we know that $A =B\cdot C$ also holds. Now, we wish to prove that $A * A' =1$. By Proposition \ref{prop:A'}, the set
\[ S= \{ b*c:b\in B,\, c\in C\}\cup \{c*b: c\in C,\, b\in B\}\]
generates $A'$ as a subgroup of $(A,{\cdot\,})$. It then follows from Corollary \ref{cor:gen} that it suffices to show, for all $a\in A,\, b\in B,\, c\in C$, that
\[ a * (b*c) = 1 \mbox{ and } a *(c*b) =1.\]
%which respectively is equivalent to
%\[ \lambda_a(b*c) = b*c \mbox{ and } \lambda_a(c*b) = c*b.\]
By symmetry, it is enough to show the former.

\vspace{2mm}

Since $A=B\circ C$, we may write
\[ a = y \circ x \mbox{ with }x\in B\mbox{ and }y\in C.\]
Observe that $b*c\in B$ because $B$ is a right ideal in $A$. Then $x * (b*c) = 1$ because $B$ is a trivial skew brace, and it follows from Lemma \ref{formulas1} that
\[ a * (b*c)  = (y\circ x) * (b*c) = y * (b*c).\]
Since $C$ is a normal subgroup of $(A,\circ)$ and is a trivial skew brace, we have
\[\overline{b}\circ y \circ b\in C,\,\ (\overline{b}\circ y \circ b\circ \overline{y}) * c =1,\,\ \lambda_y(c) = c.\]
Thus, from Lemmas \ref{formulas1} and \ref{formulas2}, we deduce that
\begin{align*}
\lambda_y(b*c) & = (y\circ b\circ \overline{y}) * \lambda_y(c)\\
& =(b\circ (\overline{b}\circ y\circ b \circ \overline{y})) * c\\
& = b* c.
 \end{align*}
This means that $y*(b*c)=1$, whence $a*(b*c)=1$, as desired. $\square$

\section{Proof of Theorem \ref{thm:right}}

Let $(A,{\cdot \,},\circ)$ be a skew brace such that $A = B\cdot C$ for sub-skew braces $B$ and $C$ satisfying conditions (1) -- (3) stated in the theorem. By Lemma \ref{lem:product}, we know that $A =B\circ C$ also holds. Now, we wish to show that $A' * A =1$. By Proposition \ref{prop:A'}, the set
\[ S = \{ b*c:b\in B,\, c\in C\}\cup \{c*b: c\in C,\, b\in B\}\]
generates $A'$ as a subgroup of $(A,{\cdot\,})$. It then follows from Corollary \ref{cor:gen} that it suffices to show, for all $a\in A,\, b\in B,\, c\in C$, that
\[ (b*c) * a = 1 \mbox{ and } (c*b) * a =1.\]
%which respectively is equivalent to
%\[ \lambda_{b*c}(a) = a \mbox{ and } \lambda_{c*b}(a) = a.\]
By symmetry, it is enough to show the former.

\vspace{2mm}
 
Since $A=B\cdot C$, we may write 
\[ a = xy\mbox{ with }x\in B\mbox{ and }c\in C.\]
Observe that $b*c\in C$ because $C$ is a left ideal in $A$. Then $(b*c) * y=1$ because $C$ is a trivial skew brace, and it  follows from Lemma \ref{formulas1} that
\[ (b*c) * a  =(b*c) *(xy) =  (b*c) * x.\]
Since $C$ is a trivial skew brace and is a left ideal in $A$, we have
\[ \overline{c} = c^{-1},\,  \lambda_b(c)\in C,\, \mbox{ and so } b*c = \lambda_b(c)c^{-1} = \lambda_b(c)\circ \overline{c}.\]
Since $B$ is a trivial skew brace and is a left ideal in $A$, we also have
\[  \overline{b} = b^{-1},\, \lambda_{b*c}(x)\in B,\,  \mbox{ and so }\lambda_{b*c}(x)=\lambda_{\overline{b}}(\lambda_{b*c}(x)) = \lambda_{\overline{b}\circ (b*c)}(x).\]
From the above, we deduce that
\begin{align*}
\lambda_{b*c}(x) & = \lambda_{\overline{b} \circ\lambda_b(c)\circ\overline{c}}(x)\\
& = \lambda_{(\overline{b}c)\circ \overline{c}}(x) \\
& = \lambda_{c \circ ( \overline{c} \circ (b^{-1}c))\circ  \overline{c}}(x).
\end{align*}
Since $C$ is a trivial skew brace, we have
\[ \overline{c} \circ (b^{-1}c) =  c^{-1}\lambda_{\overline{c}}(b^{-1}c) = c^{-1} \lambda_{\overline{c}}(b^{-1}) c.\]
But $B$ is a left ideal in $A$ and is a normal subgroup of $(A,{\cdot \, })$, so then
\[ \lambda_{\overline{c}}(b^{-1}),\overline{c} \circ (b^{-1}c),\lambda_{\overline{c}}(x)\in B,\mbox{ and }\lambda_{\overline{c} \circ (b^{-1}c)}(\lambda_{\overline{c}}(x)) =\lambda_{\overline{c}}(x)\]
because $B$ is a trivial skew brace. We now conclude that
\[ \lambda_{b*c}(x) =\lambda_{c \circ ( \overline{c} \circ (b^{-1}c))\circ  \overline{c}}(x) = \lambda_c( \lambda_{\overline{c}}(x)) = x.\]
This means that $(b*c)*x =1$, whence $(b*c)*a=1$, as desired. $\square$

 \section{Some preliminary calculations}\label{part II}
 
In this section, let $A=(A,{\cdot \,},\circ)$ be a skew brace, and let $A^{\op}=(A,{\cdot^{\op} \,},\circ)$ denote its opposite skew brace. Also let $B$ and $C$ be sub-skew braces of $A$. 

\vspace{2mm}

To simplify calculations, we shall make the following assumption.

\begin{assumption}In what follows, we shall always assume that both $B$ and $C$ are trivial skew braces. Note that then
\[ \begin{cases}
 \overline{b}= b^{-1},\, b*x=1,\, \lambda_b(x) = x,\, \mbox{and }\lambda_b^{\op}(x) = bxb^{-1} & \mbox{for all }b,x\in B,\\
 \overline{c} = c^{-1},\, c*y = 1,\, \lambda_c(y) = y,\, \mbox{and }\lambda_c^{\op}(y) = cyc^{-1} & \mbox{for all }c,y\in C.
 \end{cases}\]
We shall use these facts repeatedly without explicit mention.
\end{assumption}

To prove Theorem \ref{thm:ito}, we shall need to show that
\[ (b*c)*(x*y) = 1\mbox{ and }(b*c)*(y*x)=1\]
for all $b,x\in B$ and $c,y\in C$. Our strategy is to first write
\[ b*c = c^{-1} \circ b_1^{-1} \circ c_2\circ b_2\mbox{ with }b_1,b_2\in B\mbox{ and }c_2\in C,\]
so that for any $a\in A$, we have
\begin{equation}\label{convert} 
(b*c)*a =1 \iff \lambda_{b*c}(a) = a\iff \lambda_{b_1\circ c}(a) = \lambda_{c_2\circ b_2}(a).\end{equation}
We shall then investigate conditions under which such an equality holds for $a = x*y$ or $a=y*x$ with $x\in B$ and $y\in C$.

\begin{lem}\label{lem:b*c}For any $b\in B$ and $c\in C$, we have
\[ b*c = c^{-1}\circ b_1^{-1}\circ c_2\circ b_2, \]
where we define
\[  b_1 = \lambda_c^{\op}(b)^{-1},\,\ c_2 = \lambda^{\op}_{b_1\circ c\circ b}(c),\,\ b_2 = \lambda_{\overline{c_2}\circ b_1}^{\op}(b_1).\]
In particular, if $B$ and $C$ are left ideals in $A^{\op}$, then
\begin{equation}\label{def}
 b_1 = \lambda_c^{\op}(b)^{-1},\,\ c_2 = \lambda^{\op}_{b_1}(c\lambda_b^{\op}(c)c^{-1}),\,\ b_2 = \lambda_{\overline{c_2}}^{\op}(b_1)\end{equation}
with $b_1,b_2\in B$ and $c_2\in C$.
%\[ c_b = \lambda_b^{\op}(c),\,\ b_c = \lambda_c^{\op}(b^{-1}),\,\ d = \lambda^{\op}_{b_c}(cc_bc^{-1}),\,\ e = \lambda_{\overline{d}}^{\op}(b_c)\]
\end{lem}
\begin{proof} Using the identity $xy = y\circ \lambda_{\overline{y}}^{\op}(x)$ repeatedly, we compute that
\begin{align*}
    b*c & = b^{-1}\lambda_{b}^{\op}(c) bc^{-1}\\
    & = c^{-1} \circ \lambda_c^{\op}(b^{-1}\lambda_{b}^{\op}(c)b)\\
    & = c^{-1}\circ (b_1 \lambda_{c\circ b}^{\op}(c) b_1^{-1})\\
    & = c^{-1}\circ b_1^{-1}\circ \lambda_{b_1}^{\op}(b_1 \lambda_{c\circ b}^{\op}(c)) \\
    & = c^{-1} \circ b_1^{-1} \circ (\lambda_{b_1}^{\op}(b_1)c_2)\\
    & = c^{-1}\circ b_1^{-1} \circ c_2\circ \lambda^{\op}_{\overline{c_2}}( \lambda_{b_1}^{\op}(b_1))\\
    & = c^{-1} \circ b_1^{-1} \circ c_2\circ b_2,
    \end{align*}
as claimed. When $B$ and $C$ are left ideals in $A^{\op}$, it is clear that $b_1,b_2\in B$ and $c_2\in C$. We also have $\lambda_b^{\op}(c)\in C$. It follows that
\begin{align*}
 \lambda_{c\circ b}^{\op}(c) & = \lambda_c^{\op}(\lambda_b^{\op}(c)) = c \lambda_b^{\op}(c)c^{-1}\\
 \lambda_{ b_1}^{\op}(b_1) & = b_1b_1b_1^{-1} = b_1
 \end{align*}
because $B$ and $C$ are trivial skew braces. This yields (\ref{def}).
\end{proof}

\begin{lem}\label{lem:coset}If $C$ is a left ideal in $A^{\op}$, then 
\[ c\circ a \circ \overline{c} \in (\lambda_c^{\op}(a)c  )\circ C\mbox{ and }a\circ C =Ca\]
for all $a\in A$ and $c\in C$.
\end{lem}
\begin{proof}Put $a_0 = \lambda_c^{\op}(a)c$ for short. We have
\begin{align*}
c\circ a\circ \overline{c} & = \lambda_c^{\op}(a\circ \overline{c}) c\\
&= \lambda_c^{\op}( \lambda_a^{\op}(\overline{c})a)c\\
& = \lambda_{c\circ a }^{\op}(\overline{c})\cdot \lambda_c^{\op}(a)c\\
& = (\lambda_c^{\op}(a)c)\circ \lambda_{\overline{a_0}\circ c\circ a}^{\op}(\overline{c}),
\end{align*}
and $\lambda_{\overline{a_0}\circ c\circ a}^{\op}(\overline{c})\in C$ because $C$ is a left ideal in $A^{\op}$. This gives the first claim. For any $y\in C$, observe that
\[ a \circ y = \lambda_a^{\op}(y)a \mbox{ and }ya = a\circ \lambda_{\overline{a}}^{\op}(y),\]
where $\lambda_a^{\op}(y), \lambda_{\overline{a}}^{\op}(y)\in C$ again because $C$ is a left ideal in $A^{\op}$. This implies that $a\circ C = Ca$, as desired.
\end{proof}

\begin{lem}\label{lem:compare}For any $b,x\in B$ and $c,y\in C$, we have
   \begin{align*}
        \lambda_{b\circ c}(x*y)& = \lambda_{b}^{\op}(x_1)* y_2,\\
        \lambda_{c\circ b}(x*y)& = x_3  * y_2,
    \end{align*}
where $x_1,x_3\in B$ and $y_2\in C$ are any elements (if exist) such that
       \begin{align*}
        c\circ x \circ \overline{c} \in x_1\circ C,\,\
        \lambda_b(y) \in y_2B,\,\ 
        c\circ \lambda_{b}^{\op}(x)\circ\overline{c} \in x_3\circ C.
    \end{align*}
In particular, if $C$ is a left ideal in $A^{\op}$, then $x_1,x_3\in B$ may be taken to be any elements (if exist) such that
\[ \lambda_c^{\op}(x)c \in Cx_1\mbox{ and }\lambda_c^{\op}(\lambda_b^{\op}(x)) c \in Cx_3\]
in view of Lemma \ref{lem:coset}.
\end{lem}
\begin{proof}By the hypothesis, there exist $x_2\in B$ and $y_1,y_3\in C$ with
\[ c\circ x \circ \overline{c} = x_1\circ y_1,\, \lambda_b(y) = y_2x_2,\, 
c\circ \lambda_b^{\op}(x)\circ\overline{c}= x_3\circ y_3.\]
Notice that $\lambda_{b}^{\op}(x_1)=b\circ x_1\circ \overline{b}$ and $ \lambda_b^{\op}(x)=b\circ x \circ \overline{b}$ because $B$ is a trivial skew brace. Certainly they are both elements of $B$.
Using Lemmas \ref{formulas1} and \ref{formulas2}, we then compute that
\begin{align*}
    \lambda_{b\circ c}(x*y) & = 
    \lambda_b( (c\circ x \circ \overline{c}) * \lambda_c(y)) \\
    & = \lambda_{b}( (x_1\circ y_1) * y)&(\mbox{since $C$ is trivial})  \\
    & = \lambda_{b} (x_1 * y) &(\mbox{since }y_1*y=1)\\
    & = (b\circ x_1\circ \overline{b}) *\lambda_b(y)\\
    & = \lambda_b^{\op}(x_1) * (y_2x_2)  \\
    & = \lambda_b^{\op}(x_1) * y_2 &(\mbox{since }\lambda_b^{\op}(x_1)*x_2=1),
\end{align*}
and similarly that
\begin{align*}
    \lambda_{c\circ b}(x*y) & = \lambda_c((b\circ x \circ \overline{b}) * \lambda_b(y))\\
    & = \lambda_{c}(\lambda_b^{\op}(x) * (y_2x_2))\\
    & = \lambda_{c}(\lambda_b^{\op}(x) * y_2)&(\mbox{since }\lambda_b^{\op}(x)*x_2=1)\\
    &= (c\circ \lambda_b^{\op}(x)\circ\overline{c}) * \lambda_c(y_2)\\
    & = (x_3\circ y_3) * y_2 &(\mbox{since $C$ is trivial}) \\
    & = x_3 * y_2 & (\mbox{since $y_3*y_2=1$}).
\end{align*}
This proves the claims.
\end{proof}

The next lemma  shall also be useful.

\begin{lem}\label{lem:ideal}If $C$ is a left ideal in $A^{\op}$ and $B$ is a right ideal in $A^{\op}$, then 
\[\lambda_b^{\op}(c)c^{-1} \in B\cap C\]
 for all $b\in B$ and $c\in C$.
\end{lem}
\begin{proof}That $\lambda_b^{\op}(c)c^{-1}\in C$ holds because $C$ is a left ideal in $A^{\op}$. Since 
\begin{align*}
\lambda_b^{\op}(c)c^{-1}= (c\lambda_b^{\op}(c^{-1}))^{-1}
= ((c^{-1})^{-1} \lambda_b^{\op}(c^{-1}))^{-1}
\end{align*}
and $B$ is a right ideal in $A^{\op}$, we see that this element also lies in $B$. Thus, we indeed have $ \lambda_b^{\op}(c)c^{-1}\in B\cap C$.
\end{proof}

We now move on to the proof of Theorem \ref{thm:ito}.

\section{Proof of Theorem \ref{thm:ito}}

Let $(A,{\cdot \,},\circ)$ be a skew brace such that $A = B\cdot C$ and $A=B\circ C$ (recall Remark \ref{remark intro}) for sub-skew braces $B$ and $C$ satisfying conditions (1) and (2) stated in the theorem. Now, we wish to show that $A'$ is a trivial skew brace, or equivalently $A' * A' =1$. By Proposition \ref{prop:A'}, the set
\[ S= \{ b*c:b\in B,\, c\in C\}\cup \{c*b: c\in C,\, b\in B\}\]
generates $A'$ as a subgroup of $(A,{\cdot\,})$. It then follows from Corollary \ref{cor:gen} that it suffices to show, for all $b,x\in B$ and $c,y\in C$, that
\[ \begin{cases}
(b*c) * (x*y) = 1\\
(b*c) * (y*x) = 1
\end{cases}
\mbox{ and }\,\ 
\begin{cases}
(c*b) * (y*x) = 1,\\
(c*b) * (x*y) = 1.
\end{cases}\]%which respectively is equivalent to
%\[ \lambda_{b*c}(a) = a \mbox{ and } \lambda_{c*b}(a) = a.\]
By symmetry, it is enough to show the former set of equalities.

\vspace{2mm}

In what follows, let $b,x\in B$ and $c,y\in C$. By Lemma \ref{lem:b*c}, we may write
\[ b*c = c^{-1}\circ b_1^{-1}\circ c_2\circ b_2,\]
where $b_1,b_2\in B$ and $c_2\in C$ are given as in (\ref{def}). Then
\begin{equation}\label{eqn}
\begin{cases}
(b*c) * (x*y) = 1 \iff  \lambda_{b_1\circ c}(x*y) = \lambda_{c_2\circ b_2}(x*y)\\
(b*c)*(y*x) =1  \iff  \lambda_{b_1\circ c}(y*x) = \lambda_{c_2\circ b_2}(y*x)
\end{cases}
\end{equation}
as noted in (\ref{convert}). %We shall prove the equalities on the right using Lemmas \ref{lem:compare} and \ref{lem:coset}. 
To simplify notation, let us also put 
\[ d = \lambda_b^{\op}(c)c^{-1},\,\
e = \lambda_c^{\op}(b)^{-1} b,\,\
f = c_2c^{-1},\,\
g = b_2b.\]
We may then rewrite (\ref{def}) as 
\[ b_1 = eb^{-1},\,\ c_2 = \lambda^{\op}_{b_1}(cd),\,\ b_2 =  \lambda_{\overline{c_2}}^{\op}(b_1) =\lambda_{\overline{c_2}}^{\op}(eb^{-1}).\]
It shall be very important in the forthcoming calculations that $d,e,f,g$ are elements of both $B$ and $C$.

\begin{lem}We have $d,e,f,g\in B\cap C$.
\end{lem}
\begin{proof}Note that $d\in B\cap C$ by Lemma \ref{lem:ideal}, while $e = (b^{-1}\lambda_c^{\op}(b))^{-1}\in B\cap C$  because $B$ is a left ideal in $A^{\op}$ and $C$ is a right ideal in $A^{\op}$. It is also clear that $f\in C$ and $g\in B$.

\vspace{2mm}

To show that $f\in B$, first note that since $d\in C$, we have
\[ \lambda_{b_1}^{\op}(d)\in C\mbox{ and so }c\lambda_{b_1}^{\op}(d)c^{-1} = \lambda_{c}^{\op}(\lambda_{b_1}^{\op}(d))\]
because $C$ is a left ideal in $A^{\op}$ and is a trivial skew brace. Then
\begin{align*}
f &= c_2c^{-1}\\
&=\lambda_{b_1}^{\op}(cd)c^{-1} \\
& = \lambda_{b_1}^{\op}(c)c^{-1} \cdot c\lambda_{b_1}^{\op}(d)c^{-1}\\
& = \lambda_{b_1}^{\op}(c)c^{-1} \cdot \lambda^{\op}_{c\circ b_1}(d).
\end{align*}
But the first factor lies in $B$ by Lemma \ref{lem:ideal}. The second factor also lies in $B$ because $d\in B$ and $B$ is a left ideal in $A^{\op}$. Thus, their product lies in $B$.

\vspace{2mm}

To show that $g\in C$, we first rewrite
\begin{align*}
g & = b_2b\\
&=\lambda_{\overline{c_2}}^{\op}(eb^{-1})b\\
&=\lambda_{\overline{c_2}}^{\op}(b_1)b_1^{-1} \cdot e
%& = \lambda_{\overline{c_2}}^{\op}(b_1)b_1^{-1}\cdot (b^{-1}\lambda_c^{\op}(b))^{-1}.
\end{align*}
The first factor lies in $C$ by Lemma \ref{lem:ideal} and we already know that $e\in C$. It follows that their product also lies in $C$.
\end{proof}

The next lemma shall also be helpful.

\begin{lem}\label{lem pre1}We have the equalities
\begin{align*}
 \lambda_{b_1}^{\op}(y) &=  e\lambda^{\op}_{\overline{b}}(y)e^{-1},\\
  \lambda_{c_2}^{\op}(x) &= f\lambda_c^{\op}(x)f^{-1},\\
   \lambda_{b_2}^{\op}(y) &= g\lambda_{\overline{b}}^{\op}(y)g^{-1}.
   \end{align*} 
\end{lem}
\begin{proof} 
Since $e,f,g\in B\cap C$ and $B,C$ are left ideals in $A^{\op}$, we know that
\[ \lambda_{\overline{c}}^{\op}(f) \in B\mbox{ and }\lambda_{b}^{\op}(e),\lambda_b^{\op}(g)\in C.\]
Since $B$ and $C$ are trivial skew braces, we then have
\[\lambda_{\lambda_{\overline{c}}^{\op}(f)} (x) = x,\,\  \lambda_{\lambda_b^{\op}(e)}(y) = y,\,\ \lambda_{\lambda_b^{\op}(g)}(y) = y.\]
It shall also be helpful to note that $\lambda$ and $\lambda^{\op}$ are related by
\[ \lambda^{\op}_{a}(z) = a\lambda_a(z)a^{-1}  \mbox{ for any }a,z\in A.\]
Using $b_1 = eb^{-1} = \overline{b} \circ \lambda_{b}^{\op}(e)$, we compute that 
\begin{align*}
\lambda_{b_1}^{\op}(y) & = b_1\lambda_{b_1}(y) {b_1}^{-1}\\
& = eb^{-1} \lambda_{\overline{b} \circ \lambda_{b}^{\op}(e)}(y) be^{-1}\\
& = eb^{-1}\lambda_{\overline{b}}(y) be^{-1} \\
& = e\lambda^{\op}_{\overline{b}}(y)e^{-1}.
\end{align*}
Similarly, using $ c_2 = fc = c\circ\lambda_{\overline{c}}^{\op}(f)$, we compute that
\begin{align*}
\lambda_{c_2}^{\op}(x) & = c_2\lambda_{c_2}(x)c_2^{-1}\\
&= fc \lambda_{c\circ \lambda_{\overline{c}}^{\op}(f)}(x) c^{-1}f^{-1}\\
&=fc \lambda_{c }(x) c^{-1}f^{-1}  \\
& =f \lambda_c^{\op}(x) f^{-1}.
\end{align*}
Finally, using $b_2 =gb^{-1} = \overline{b}\circ \lambda_b^{\op}(g)$, we compute that
 \begin{align*}
 \lambda_{b_2}^{\op}(y) & = b_2\lambda_{b_2}(y)b_2^{-1}\\
 & = gb^{-1}\lambda_{\overline{b}\circ \lambda_b^{\op}(g)}(y) bg^{-1}\\
 & =gb^{-1}\lambda_{\overline{b}}(y) bg^{-1}\\
 & = g\lambda_{\overline{b}}^{\op}(y)g^{-1}.
  \end{align*}
  The completes the proof.
  \end{proof}

We are now ready to prove the two equalities
\begin{equation}\label{eqn'}
\begin{cases}
\lambda_{b_1\circ c}(x*y) = \lambda_{c_2\circ b_2}(x*y)\\
\lambda_{b_1\circ c}(y*x) = \lambda_{c_2\circ b_2}(y*x)
\end{cases}
\end{equation}
in (\ref{eqn}), from which Theorem \ref{thm:ito} would follow. From Lemma \ref{lem:compare} (with the roles of $B$ and $C$ swapped for the equalities on the right), we see that
\begin{align*}
\lambda_{b_1\circ c}(x*y)&= \lambda_{b_1}^{\op}(x_1) * y_2,  & \lambda_{c_2\circ b_2}(y*x)&= \lambda_{c_2}^{\op}(y_3') * x_4',\\
 \lambda_{c_2\circ b_2}(x*y)&= x_1' * y_2',&\lambda_{b_1\circ c}(y*x)&= y_3 * x_4.
\end{align*}
Here $x_1,x_1'\in B$ and $y_2,y_2'\in C$ are any elements (which must exist because of the hypothesis $A=B\cdot C$) such that
\[ \begin{cases}
\lambda_{c}^{\op}(x)c \in  Cx_1 \\
\lambda_{b_1}(y) \in y_2B
\end{cases}\mbox{ and }\,\ \,
\begin{cases}
\lambda_{c_2}^{\op}(\lambda_{b_2}^{\op}(x))c_2\in Cx_1',\\
\lambda_{b_2}(y)\in y_2'B,
\end{cases}\]
while $y_3,y_3'\in C$ and $x_4,x_4'\in B$ are any elements (which must exist because\par\noindent of the hypothesis $A=B\cdot C$) such that
\[ \begin{cases}
\lambda_{b_2}^{\op}(y)b_2 \in  By_3' \\
\lambda_{c_2}(x) \in x_4'C
\end{cases}\mbox{ and }\,\
\begin{cases}
\lambda_{b_1}^{\op}(\lambda_{c}^{\op}(y))b_1\in By_3,\\
\lambda_{c}(x)\in x_4C.
\end{cases}\]
But as Lemmas \ref{lem1} and \ref{lem2} below show, we may take 
\[x_1' = \lambda_{b_1}^{\op}(x_1),\,\ y_2' = y_2,\,\ y_3' = \lambda_{\overline{c_2}}^{\op}(y_3),\,\ x_4' = x_4.\]
It follows that the equalities in (\ref{eqn'}) indeed hold.

\begin{lem}\label{lem1}In the above notation, we have 
\[\lambda_{b_2}(y)  =\lambda_{b_1}(y)\mbox{ and }
\lambda_{c_2}^{\op}(\lambda_{b_2}^{\op}(x))c_2 \in C\lambda^{\op}_{b_1}(x_1).\]
\end{lem}

\begin{proof}Since $B$ is a trivial skew brace, we have
\[ \overline{b_1}\circ b_2 = b_1^{-1}b_2 = b_1^{-1}\lambda_{\overline{c_2}}^{\op}(b_1),\]
and this lies in $C$ because $C$ is a right ideal in $A^{\op}$. Since $C$ is a trivial skew brace, we then get that $\lambda_{\overline{b_1}\circ b_2}(y) =y$, whence $\lambda_{b_2}(y)  =\lambda_{b_1}(y)$.

\vspace{2mm}

Using Lemma \ref{lem pre1}, we first rewrite
\begin{align*}
\lambda_{c_2}^{\op}(\lambda_{b_2}^{\op}(x))c_2 
&= \lambda_{c_2}^{\op}(b_2xb_2^{-1})c_2  &(\mbox{since $B$ is trivial})\\
& = b_1\lambda_{c_2}^{\op}(x) b_1^{-1}\lambda_{b_1}^{\op}(cd)\\
& = b_1f\lambda_c^{\op}(x) f^{-1} b_1^{-1} \lambda_{b_1}^{\op}(cd).
\end{align*}
Note that $\lambda_c^{\op}(x)\in B$ because $B$ is a left ideal in $A^{\op}$. Since $f\in B$ and  $B$ is a trivial skew brace, the above simplifies to
\begin{align*}
\lambda_{c_2}^{\op}(\lambda_{b_2}^{\op}(x))c_2 
& = \lambda_{b_1}^{\op}(f\lambda_c^{\op}(x) f^{-1})\lambda_{b_1}^{\op}(cd ).
%&=\lambda_{b_1}^{\op}(f\lambda_c^{\op}(x) f^{-1}cd).
\end{align*}
Since $\lambda_{c}^{\op}(x)c \in  Cx_1$ by choice, we may write
\[ \lambda_{c}^{\op}(x)c = y_1x_1\mbox{ or equivalently }\lambda_c^{\op}(x) = y_1x_1c^{-1}\mbox{ with }y_1\in C.\]
Since $f\in C$ and $C$ is a trivial skew brace, we obtain
\begin{align*}
\lambda_{c_2}^{\op}(\lambda_{b_2}^{\op}(x))c_2 
& = \lambda_{b_1}^{\op}(f y_1x_1\cdot c^{-1} f^{-1}c\cdot  d)\\
& = \lambda_{b_1}^{\op}(f y_1x_1\cdot \lambda_{\overline{c}}^{\op}(f^{-1})\cdot d)\\
& = \lambda_{b_1}^{\op}(f y_1\cdot x_1\lambda_{\overline{c}}^{\op}(f^{-1})dx_1^{-1})\cdot \lambda_{b_1}^{\op}(x_1).
\end{align*}
But $f\in B$ and $B$ is a left ideal in $A^{\op}$, so we have $\lambda_{\overline{c}}^{\op}(f^{-1})\in B$. Since $d\in B$ also and $B$ is a trivial skew brace, we get that
\begin{align*}
\lambda_{c_2}^{\op}(\lambda_{b_2}^{\op}(x))c_2 
& = \lambda_{b_1}^{\op}(f y_1\cdot  \lambda_{x_1}^{\op}(\lambda_{\overline{c}}^{\op}(f^{-1})d))\cdot\lambda_{b_1}^{\op}(x_1).
%& = \lambda_{b_1}^{\op}(f y_1\lambda_{x_1\circ \overline{c}}^{\op}(f^{-1})\lambda_{x_1}^{\op}(d))\cdot\lambda_{b_1}^{\op}(x_1).
\end{align*}
But $f,d\in C$ also and $C$ is a left ideal in $A^{\op}$. This implies that
\[\lambda_{b_1}^{\op}(f y_1\lambda_{x_1}^{\op}(\lambda_{\overline{c}}^{\op}(f^{-1})d))\in C,\]
and the desired containment follows.
\end{proof}
 
\begin{lem}\label{lem2}In the above notation, we have
\[ \lambda_{c_2}(x) = \lambda_c(x)\mbox{ and }
\lambda_{b_2}^{\op}(y)b_2 \in B\lambda_{\overline{c_2}}^{\op}(y_3).\]
\end{lem}
\begin{proof}Since $f\in C$ and $C$ is a trivial skew brace, we have
\begin{equation}\label{above} \overline{c} \circ c_2 = c^{-1}c_2 = c^{-1}fc = \lambda_{\overline{c}}^{\op}(f).\end{equation}
But $f\in B$ also and $B$ is a left ideal in $A^{\op}$. It follows that $\overline{c} \circ c_2\in B$. Since \par\noindent $B$ is a trivial skew brace, we get that $\lambda_{\overline{c}\circ c_2}(x) = x$, whence $\lambda_{c_2}(x)  = \lambda_{c}(x)$.

\vspace{2mm}

Note that $\lambda_c^{\op}(y) =cyc^{-1}$ because $C$ is a trivial skew brace. By Lemma \ref{lem pre1} (with $y$ replaced by $\lambda_c^{\op}(y)$), we know that
\begin{align*}
\lambda_{b_1}^{\op}(\lambda_c^{\op}(y))b_1
& = e\lambda_{\overline{b}}^{\op}(\lambda_c^{\op}(y))e^{-1}b_1\\
& = e\lambda_{\overline{b}}^{\op}(\lambda_c^{\op}(y))\lambda_c^{\op}(y)^{-1} \cdot cyc^{-1}b^{-1}.
\end{align*}
Since $\lambda_{\overline{b}}^{\op}(\lambda_{c}^{\op}(y))\lambda_{c}^{\op}(y)^{-1} \in B$ by Lemma \ref{lem:ideal} and $e\in B$ also, we see that
\[cyc^{-1}b^{-1} \in B\lambda_{b_1}^{\op}(\lambda_{c}^{\op}(y))b_1,\]
which is equal to $By_3$ by choice. Thus, we may write
\[ cyc^{-1}b^{-1} = x_3y_3\mbox{ or equivalently }b^{-1}=cy^{-1}c^{-1}x_3y_3\mbox{ with }x_3\in B.\]
Now, using Lemma \ref{lem pre1}, we compute that
\begin{align*}
\lambda_{b_2}^{\op}(y)b_2 & = g\lambda_{\overline{b}}^{\op}(y)g^{-1} \lambda_{\overline{c_2}}^{\op}(eb^{-1})\\
& = g\lambda_{\overline{b}}^{\op}(y)y^{-1}\cdot yg^{-1}\cdot \lambda_{\overline{c_2}}^{\op}(ecy^{-1}c^{-1}x_3y_3).
\end{align*}
Since $g\in C$ and $C$ is a trivial skew brace, we have
\[ \lambda_{c_2}^{\op}(yg^{-1}) = c_2yg^{-1}c_2^{-1} \mbox{ or equivalently }yg^{-1} = \lambda_{\overline{c_2}}^{\op}(c_2yg^{-1}c_2^{-1}).\]
Recalling that $c_2 = fc$, we then obtain
\begin{align*}
\lambda_{b_2}^{\op}(y)b_2 
&= g\lambda_{\overline{b}}^{\op}(y)y^{-1}\cdot  \lambda_{\overline{c_2}}^{\op}(fcyg^{-1} c^{-1}f^{-1}\cdot e cy^{-1}c^{-1}x_3y_3 )\\
&=g\lambda_{\overline{b}}^{\op}(y)y^{-1} \cdot \lambda_{\overline{c_2}}^{\op}( f\cdot cy(g^{-1}\cdot c^{-1}f^{-1}ec)y^{-1}c^{-1}\cdot x_3)\cdot \lambda^{\op}_{\overline{c_2}}(y_3).
\end{align*}
But $e,f,g\in C$ and $C$ is a trivial skew brace, so the above becomes
\begin{align*}
 \lambda_{b_2}^{\op}(y)b_2 & =g\lambda_{\overline{b}}^{\op}(y)y^{-1} \cdot \lambda_{\overline{c_2}}^{\op}( f\cdot \lambda_{cy}^{\op}(g^{-1}\cdot c^{-1}f^{-1}ec)\cdot x_3)\cdot \lambda^{\op}_{\overline{c_2}}(y_3)\\
 & =g\lambda_{\overline{b}}^{\op}(y)y^{-1} \cdot \lambda_{\overline{c_2}}^{\op}( f\cdot \lambda_{cy}^{\op}(g^{-1}\lambda_{\overline{c}}^{\op}(f^{-1}e))\cdot x_3)\cdot \lambda^{\op}_{\overline{c_2}}(y_3).
 \end{align*}
But $e,f,g\in B$ also. We then deduce from Lemma \ref{lem:ideal} that
\[g\lambda_{\overline{b}}^{\op}(y)y^{-1}\in B,\mbox{ and } \lambda_{\overline{c_2}}^{\op}( f \lambda_{cy}^{\op}(g^{-1}\lambda_{\overline{c}}^{\op}(f^{-1}e))x_3)\in B\]
because $B$ is a left ideal in $A^{\op}$. The desired containment now follows.\end{proof}

This concludes the proof of Theorem \ref{thm:ito}. $\square$

\section{A construction of ditrivial skew braces} \label{construct sec}

In this section, let $B$ and $C$ be any groups with $C$ abelian. Also let
\[ \phi: C \longrightarrow \Aut(B)\mbox{ and } \psi : B \longrightarrow \Aut(C)\]
be any homomorphisms such that
\[ \phi_{\psi_b(c)} = \phi_c\mbox{ for all }b\in B\mbox{ and }c\in C.\]
Using these data, we may endow $B\times C$ with a skew brace structure.

\vspace{2mm}

 Put $A=B\times C$ and endow it with the operations ${\cdot\,}$ and $\circ$ defined by
  \begin{align*}
(b_1,c_1) \cdot (b_2,c_2) & = (b_1\phi_{c_1}(b_2),c_1c_2),\\
(b_1,c_1)\circ (b_2,c_2) & = (b_1b_2,c_1\psi_{b_1}(c_2)).
\end{align*}
Clearly $(A,{\cdot\, })$ and $(A,\circ)$ are groups such that
\begin{equation}\label{sd} (A,{\cdot\, })= B\rtimes_\phi C\mbox{ and }(A,\circ)= B\ltimes_\psi C.\end{equation}
The next lemma shows that $(A,{\cdot\, },\circ)$ is in fact a skew brace.

\begin{lem}For any $a_1,a_2,a_3\in A$, we have
\[ a_1 \circ (a_2\cdot a_3) = (a_1\circ a_2)\cdot a_1^{-1} \cdot (a_1\circ a_3). \]
\end{lem}
\begin{proof}For each $i=1,2,3$, write $a_i = (b_i,c_i)$ for $b_i\in B$ and $c_i\in C$. Then
\begin{align*}
a_1 \circ (a_2\cdot a_3) & = (b_1,c_1)\circ (b_2\phi_{c_2}(b_3),c_2c_3)\\
&= (b_1b_2\phi_{c_2}(b_3),c_1\psi_{b_1}(c_2c_3)),
\end{align*}
and on the other hand, we have
\begin{align*}
&(a_1\circ a_2)\cdot a_1^{-1} \cdot (a_1\circ a_3)\\
 &\hspace{1em}= (b_1b_2,c_1\psi_{b_1}(c_2))\cdot (\phi_{c_1^{-1}}(b_1^{-1}),c_1^{-1})\cdot (b_1b_3,c_1\psi_{b_1}(c_3))\\
 &\hspace{1em} = (b_1b_2\phi_{c_1\psi_{b_1}(c_2)c_1^{-1}}(b_1^{-1}),c_1\psi_{b_1}(c_2)c_1^{-1})\cdot  (b_1b_3,c_1\psi_{b_1}(c_3))\\
 & \hspace{1em}= (b_1b_2\phi_{\psi_{b_1}(c_2)}(b_1^{-1}),\psi_{b_1}(c_2))\cdot (b_1b_3,c_1\psi_{b_1}(c_3)) &(C\mbox{ is abelian})\\
 &\hspace{1em}=(b_1b_2\phi_{\psi_{b_1}(c_2)}(b_1^{-1})\phi_{\psi_{b_1}(c_2)}(b_1b_3),\psi_{b_1}(c_2)c_1\psi_{b_1}(c_3))\\
 &\hspace{1em} = (b_1b_2\phi_{\psi_{b_1}(c_2)}(b_3),c_1\psi_{b_1}(c_2c_3)) &(C\mbox{ is abelian}).
\end{align*}
The claim now follows since $\phi_{\psi_{b_1}(c_2)} = \phi_{c_2}$ by hypothesis.
\end{proof}

It is immediate that $(A,{\cdot\,},\circ)$ is \emph{ditrivial}, in the sense that $A$ is factorizable as the product (with respect to both $\cdot$ and $\circ$) of two trivial sub-skew braces. Indeed, clearly $B\times 1$ and $1\times C$ are sub-skew braces of $A$ which are trivial as skew braces. We also have
\[ A = (B\times 1)\cdot (1\times C)\mbox{ and }
A = (B\times 1)\circ (1\times C)\]
by (\ref{sd}). The next two propositions tell us when $B\times 1$ and $1\times C$ are left or right ideals in $A$ and $A^{\op}$, respectively.

\begin{prop}\label{prop:ideal} The following statements hold.
\begin{enumerate}[$(1)$]
\item $B\times 1$ and $1\times C$ are always left ideals in $A$.
\item $B\times 1$ is a right ideal in $A$ if and only if $\psi$ is trivial.
\item $1\times C$ is a right ideal in $A$ if and only if $\phi$ is trivial.
\end{enumerate}
\end{prop}
\begin{proof}Let $b,x\in B$ and $c,y\in C$ be arbitrary. We have
\begin{align*}
\lambda_{(b,c)}(x,y) & = (b,c)^{-1}\cdot ((b,c)\circ (x,y))\\
& = (\phi_{c^{-1}}(b^{-1}),c^{-1})\cdot (bx, c\psi_b(y))\\
&= (\phi_{c^{-1}}(x), \psi_b(y)).
\end{align*}
In particular, this implies that 
\begin{align*}
\lambda_{(b,c)}(x,1) &= (\phi_{c^{-1}}(x),1) \in B\times 1,\\
\lambda_{(b,c)}(1,y) & = (1,\psi_b(y))\in 1\times C.
\end{align*}
Thus, indeed $B\times 1$ and $1\times C$ are left ideals in $A$. Similarly, we have
\begin{align*}
\lambda_{(x,y)}(b,c) \cdot (b,c)^{-1}& = (\phi_{y^{-1}}(b), \psi_x(c))\cdot (\phi_{c^{-1}}(b^{-1}),c^{-1})\\
& = (\phi_{y^{-1}}(b)\phi_{\psi_x(c)c^{-1}}(b^{-1}),\psi_x(c)c^{-1})\\
&= (\phi_{y^{-1}}(b) b^{-1},\psi_x(c)c^{-1}),
\end{align*}
where $\phi_{\psi_x(c)} = \phi_c$ by the hypothesis. In particular, we get that
\begin{align*}
\lambda_{(x,1)}(b,c)\cdot (b,c)^{-1} &= (1,\psi_x(c)c^{-1}),\\
 \lambda_{(1,y)}(b,c) \cdot (b,c)^{-1}& = (\phi_{y^{-1}}(b)b^{-1}, 1).
\end{align*}
We then deduce that
\begin{align*}
B\times 1\mbox{ is a right ideal in $A$} & \iff \psi_x(c)=c\mbox{ for all $x\in B$ and $c\in C$},\\
1\times C\mbox{ is a right ideal in $A$} &  \iff \phi_{y}(b) = b\mbox{ for all $y\in C$ and $b\in B$}.
\end{align*}
This proves the proposition.
\end{proof}

\begin{prop}\label{prop:ideal op}The following statements hold.
\begin{enumerate}[$(1)$]
\item $B\times 1$ is always a left ideal in $A^{\op}$.
\item $1\times C$ is a left ideal in $A^{\op}$ if and only if $\phi$ is trivial.
\item $1\times C$ is always a right ideal in $A^{\op}$.
\item $B\times 1$ is a right ideal in $A^{\op}$ if and only if $\psi$ is trivial.
\end{enumerate}
\end{prop}

\begin{proof}
Let $b,x\in B$ and $c,y\in C$ be arbitrary. We have
\begin{align*}
\lambda^{\op}_{(b,c)}(x,y) & = ((b,c)\circ (x,y))\cdot (b,c)^{-1}\\
& =  (bx, c\psi_b(y))\cdot (\phi_{c^{-1}}(b^{-1}),c^{-1})\\
&= (bx\phi_{c\psi_b(y)c^{-1}}(b^{-1}),c\psi_b(y)c^{-1})\\
&= (bx\phi_{\psi_b(y)}(b^{-1}),\psi_b(y)) &(C\mbox{ is abelian})\\
&= (bx\phi_{y}(b^{-1}),\psi_b(y)),
\end{align*}
where $\phi_{\psi_b(y)}  = \phi_y$ by hypothesis. In particular, this yields
\[\lambda^{\op}_{(b,c)}(x,1)  = (bxb^{-1},1)\in B\times 1,\]
which shows that $B\times 1$ is a left ideal in $A^{\op}$. We also have
\[ \lambda^{\op}_{(b,c)}(1,y) = (b\phi_y(b^{-1}),\psi_b(y)),\]
which implies that
\begin{align*}
1\times C\mbox{ is a left ideal in $A^{\op}$} %&\iff b\phi_y(b^{-1})=1\mbox{ for all }b\in B\mbox{ and }y\in C\\
&\iff \phi_y(b) = b \mbox{ for all $y\in C$ and $b\in B$}.
\end{align*}
Similarly, we compute that
\begin{align*}
(b,c)^{-1}\cdot \lambda_{(x,y)}^{\op}(b,c) & = (\phi_{c^{-1}}(b^{-1}),c^{-1})\cdot (xb\phi_c(x^{-1}),\psi_x(c))\\
& = (\phi_{c^{-1}}(b^{-1}xb)x^{-1},c^{-1}\psi_x(c)).
\end{align*}
In particular, we get that
\[(b,c)^{-1}\cdot \lambda_{(1,y)}^{\op}(b,c) = (1,1)\in 1\times C,\]
and hence $1\times C$ is a right ideal in $A^{\op}$. We also have
\[(b,c)^{-1}\cdot \lambda_{(x,1)}^{\op}(b,c) =(\phi_{c^{-1}}(b^{-1}xb)x^{-1},c^{-1}\psi_x(c)), \]
which implies that
\begin{align*}
B\times 1\mbox{ is a right ideal in $A^{\op}$} %&\iff c^{-1}\psi_x(c^{-1})=1\mbox{ for all }x\in B\mbox{ and }c\in C\\
&\iff \psi_x(c) = c\mbox{ for all $x\in B$ and $c\in C$}.
\end{align*}
This proves the proposition.
\end{proof}

Now, since $(A,{\cdot\,},\circ)$ is ditrivial, in view of Theorem \ref{thm:ito} it is natural to ask whether it is a meta-trivial skew brace. By Proposition \ref{prop:A'}, we know that $A'$ is generated by
\[ \{(b,1) * (1,c) : b\in B,\, c\in C\}\cup \{(1,c)*(b,1) : c\in C,\, b\in B\} \]
as a subgroup of $(A,{\cdot \, })$. For any $b\in B$ and $c\in C$, we compute that
\begin{align}\label{b*c}
(b,1)*(1,c) & = (b,1)^{-1}\cdot ((b,1)\circ (1,c))\cdot (1,c)^{-1}\\\notag
& = (b^{-1},1) \cdot (b,\psi_b(c)) \cdot (1,c^{-1})\\\notag
& = (1,\psi_b(c))\cdot (1,c^{-1})\\\notag
& = (1,\psi_b(c)c^{-1}),\\\label{c*b}
(1,c) * (b,1) & = (1,c)^{-1}\cdot ((1,c)\circ (b,1))\cdot (b,1)^{-1}\\\notag
&= (1,c^{-1})\cdot (b,c) \cdot  (b^{-1},1)\\\notag
& = (\phi_{c^{-1}}(b),1)\cdot (b^{-1},1)\\\notag
& = (\phi_{c^{-1}}(b)b^{-1},1).
\end{align}
%We have thus shown that:
The next proposition provides us with a criterion for determining when $A$ is a meta-trivial skew brace.
%\begin{prop}\label{prop:A' gen}The ideal $A'$ is generated by
%\[ \{ (1,\psi_b(c)c^{-1}) : b\in B,\, c\in C\} \cup \{ (\phi_{c}(b)b^{-1},1): c\in C,\, b\in B\}.\]
%as a subgroup of $(A,{\cdot \, })$.
%\end{prop}

\begin{prop}\label{prop:criterion}The skew brace $A$ is meta-trivial if and only if
\[\psi_{\phi_{c_1}(b_1)b_1^{-1}}( \psi_{b_2}(c_2)c_2^{-1}) =  \psi_{b_2}(c_2)c_2^{-1}\]
for all $b_1,b_2\in B$ and $c_1,c_2\in C$. %In particular, if either $\phi$ or $\psi$ is trivial, then the skew brace $A$ is meta-trivial.
\end{prop}
\begin{proof}By the above discussion, we know that $A'$ is generated by $S_C\cup S_B$ as a subgroup of $(A,{\cdot \,})$, where we define
\begin{align*}
S_C & =\{ (1,\psi_b(c)c^{-1}) : b\in B,\, c\in C\} ,\\
S_B & =\{ (\phi_{c}(b)b^{-1},1): c\in C,\, b\in B\}.
\end{align*}
It then follows from Corollary \ref{cor:gen} that $A$ is meta-trivial, that is $A'*A'=1$, if and only if $x_1*x_2 =1$ for all $x_1*x_2\in S_C\cup S_B$. Notice that $S_C\subseteq 1\times C$ and $S_B\subseteq B\times 1$. Since $1\times C$ and $B\times 1$ are trivial skew braces, plainly
\[ x_1*x_2=1\mbox{ whenever }x_1,x_2\in S_C\mbox{ or }x_1,x_2\in S_B.\]
For any $b_1,b_2\in B$ and $c_1,c_2\in C$, using (\ref{c*b}), we compute that
\begin{align*}
&(1,\psi_{b_1}(c_1)c_1^{-1}) * (\phi_{c_2}(b_2)b_2^{-1},1)\\
& \hspace{1em}= (\phi_{c_1\psi_{b_1}(c_1)^{-1}}(\phi_{c_2}(b_2)b_2^{-1})(\phi_{c_2}(b_2)b_2^{-1})^{-1},1)\\
&\hspace{1em}= (1,1)
\end{align*}
because $\phi_{c_1\psi_{b_1}(c_1)^{-1}}=\mathrm{id}_B$ by hypothesis. Similarly, using (\ref{b*c}), we get that
\begin{align*}
&(\phi_{c_1}(b_1)b_1^{-1},1)* (1,\psi_{b_2}(c_2)c_2^{-1})\\
&\hspace{1em} = (1,\psi_{\phi_{c_1}(b_1)b_1^{-1}}(\psi_{b_2}(c_2)c_2^{-1})(\psi_{b_2}(c_2)c_2^{-1})^{-1})
\end{align*}
This element equals $(1,1)$ exactly when
\[ \psi_{\phi_{c_1}(b_1)b_1^{-1}}(\psi_{b_2}(c_2)c_2^{-1}) = \psi_{b_2}(c_2)c_2^{-1}\]
and the claim now follows.
\end{proof}

\begin{cor}\label{cor:metatrivial}The skew brace $A$ is meta-trivial when $\phi$ or $\psi$ is trivial.
\end{cor}

It is also easy to see directly why Corollary \ref{cor:metatrivial} is true. If $\phi$ is trivial, then $1\times C$ is clearly an ideal in $A$ and $A/(1\times C)\simeq B\times 1$ as skew braces. If $\psi$ is trivial, then $B\times 1$ is an ideal in $A$ and $A/(B\times 1)\simeq 1\times C$ as skew braces. Since $B\times 1$ and $1\times C$ are trivial skew braces, in both cases we get that $A$ is meta-trivial.

%Let us also remark that
%\begin{align*}
%\phi\mbox{ is trivial}& \iff 1\times C\mbox{ is a right ideal in $A$}\\
%&\iff1\times C\mbox{ is a left ideal in $A^{\op}$}\\
%\psi\mbox{ is trivial}& \iff B\times 1\mbox{ is a right ideal in $A$}\\
%&\iff B\times 1\mbox{ is a right ideal in $A^{\op}$}
%\end{align*}
%by Propositions \ref{prop:ideal} and \ref{prop:ideal op}.

\section{Two families of ditrival skew braces}\label{ex sec}

We now apply the construction described in Section \ref{construct sec} to exhibit examples of skew braces which are factorizable as the product of two trivial sub-skew braces. We give two families of examples -- one consists of meta-trivial skew braces while the other consists of non-meta-trivial skew braces.

\subsection{First family}

Let $p$ be any odd prime. Take 
\[ B = \mathbb{Z}/p^m\mathbb{Z}\mbox{ and }C = \mathbb{Z}/p^n\mathbb{Z},\]
where $m$ and $n$ are any natural numbers. Also take
\begin{align*}
\phi : \mathbb{Z}/p^n \mathbb{Z}\longrightarrow \Aut(\mathbb{Z}/p^m\mathbb{Z});&\,\ \phi_c(x) = (1+p^{m-k})^cx,\\
\psi : \mathbb{Z}/p^m\mathbb{Z}\longrightarrow \Aut(\mathbb{Z}/p^n\mathbb{Z});&\,\ \psi_b(y) = (1+p^{n-\ell})^by,
\end{align*}
where $k$ and $\ell$ are any natural numbers satisfying
\begin{equation}\label{ineq}
k\leq \min\{m,n\},\,\ \ell \leq \min\{m,n\},\,\ k\leq n-\ell.
\end{equation}
The multiplicative orders of $(1+p^{m-k})$ mod $p^m$ and $(1+p^{n-\ell})$ mod $p^n$ are equal to $p^k$ and $p^\ell$, respectively. The conditions $k\leq n$ and $\ell \leq m$ are here to ensure that $\phi$ and $\psi$ are well-defined. As for the last condition $k\leq n-\ell$, it is imposed in order to prove the following lemma.

\begin{lem}\label{lem phi}For any $b\in \mathbb{Z}/p^m\mathbb{Z}$ and $c\in \mathbb{Z}/p^n\mathbb{Z}$, we have $\phi_{\psi_b(c)}=\phi_c$.
\end{lem}
\begin{proof}The automorphism $\phi_{\psi_b(c)-c}$ of $\mathbb{Z}/p^m\mathbb{Z}$ is multiplication by
\[(1+p^{m-k})^{\psi_b(c)-c} = (1+p^{m-k})^{((1+p^{n-\ell})^{b}-1)c}.\]
Since $p^{n-\ell}$ divides $(1+p^{n-\ell})^{b}-1$ and $(1+p^{m-k})$ mod $p^m$ has multiplicative order $p^k$, the assumption $k\leq n-\ell$ implies that 
\[ (1+p^{m-k})^{\psi_b(c)-c} = 1.\]
This means that $\phi_{\psi_b(c)-c}$ is the identity, as desired.
\end{proof}

Lemma \ref{lem phi} implies that the construction in Section \ref{construct sec} applies. Hence, the set $A=\mathbb{Z}/p^m\mathbb{Z}\times\mathbb{Z}/p^n\mathbb{Z}$ endowed with the operations
  \begin{align*}
(b_1,c_1) \cdot (b_2,c_2) & = (b_1 + (1+p^{m-k})^{c_1}b_2,c_1+c_2)\\
(b_1,c_1)\circ (b_2,c_2) & = (b_1+b_2,c_1+(1+p^{n-\ell})^{b_1}c_2)
\end{align*}
becomes a skew brace. Moreover, we have
\begin{align*}
 A &= (\mathbb{Z}/p^m\mathbb{Z}\times 0)\cdot (0\times \mathbb{Z}/p^n\mathbb{Z}),\\
A &= (\mathbb{Z}/p^m\mathbb{Z}\times 0)\circ (0\times \mathbb{Z}/p^n\mathbb{Z}),\end{align*}
where both of the factors are sub-skew braces of $A$ which are trivial as skew braces. It turns out that $A$ is always meta-trivial.

\begin{prop}The skew brace $A$ is meta-trivial.
\end{prop}
\begin{proof}For any $b_1,b_2\in \mathbb{Z}/p^m\mathbb{Z}$ and $c_1,c_2\in \mathbb{Z}/p^n\mathbb{Z}$, we have
\begin{align*}
& \psi_{\phi_{c_1}(b_1)-b_1}(\psi_{b_2}(c_2)-c_2) - (\psi_{b_2}(c_2)-c_2)\\
 &\hspace{1em}= ((1+p^{n-\ell})^{\phi_{c_1}(b_1)-b_1}-1)(\psi_{b_2}(c_2)-c_2)\\
 &\hspace{1em}= ((1+p^{n-\ell})^{((1+p^{m-k})^{c_1}-1)b_1}-1)((1+p^{n-\ell})^{b_2}-1)c_2.
 \end{align*}
For the first factor, since $p^{m-k}$ divides $(1+p^{m-k})^{c_1}-1$, we see that
\[(1+p^{n-\ell})^{((1+p^{m-k})^{c_1}-1)b_1}\equiv 1\hspace{-2mm}\pmod{p^{n-\ell + m -k}}.\]
For the second factor, it is clear that
\[(1+p^{n-\ell})^{b_2}\equiv 1\hspace{-2mm}\pmod{p^{n-\ell}}.\]
It then follows that
\[\psi_{\phi_{c_1}(b_1)-b_1}(\psi_{b_2}(c_2)-c_2) - (\psi_{b_2}(c_2)-c_2) \equiv 0\hspace{-2mm} \pmod{p^{2(n-\ell) + (m-k)}}.\]
But by the assumption (\ref{ineq}), we have
\[ 2(n-\ell) + (m-k) = n + (m-\ell) + (n-\ell - k)\geq n,\]
and we deduce from Proposition \ref{prop:criterion} that $A$ is meta-trivial.
\end{proof}

As we noted after Corollary \ref{cor:metatrivial}, it is obvious that $A$ is a meta-trivial skew brace when $\phi$ or $\psi$ is trivial, or equivalently, when $k=m$ or $\ell =n$. But it is possible that $k\leq m-1$ and $\ell \leq n-1$, so our construction here does yield non-trivial examples. For $1\leq m, n,k,\ell \leq 3$, the quadruples 
\[(2, 2, 1, 1),\, (2, 3, 1, 1), \, (2, 3, 1, 2),\, (3, 2, 1, 1),\]
\[(3, 3, 1, 1), \, (3, 3,1, 2),\,  (3, 3, 2, 1)\]
are examples of $(m,n,k,\ell)$ with $k\leq m-1$ and $\ell\leq n-1$ which satisfy (\ref{ineq}). %For $1\leq m, n,k,\ell \leq 5$, there are $61$ such quadruples. 
For $1\leq m, n,k,\ell \leq 10$, there are $1025$ such quadruples.

\subsection{Second family} Let $p$ be any odd prime. Take
\[ B = (\mathbb{Z}/2\mathbb{Z})^m \mbox{ and }C = (\mathbb{Z}/p\mathbb{Z})^n,\]
where $m$ and $n$ are any natural numbers such that $n\geq 2$ and
\[ |\mathrm{GL}_m(\mathbb{Z}/2\mathbb{Z})| = 2^{m\choose 2} \prod_{i=1}^{m}(2^i - 1)\]
is divisible by $p$. Also take
\begin{align*}
\phi: (\mathbb{Z}/p\mathbb{Z})^n\longrightarrow \Aut((\mathbb{Z}/2\mathbb{Z})^m);&\,\ \phi_{(v_1,\dots,v_n)}(\vec{x}) = P^{v_1}\vec{x},\\
\psi: (\mathbb{Z}/2\mathbb{Z})^m \longrightarrow\Aut((\mathbb{Z}/p\mathbb{Z})^n);&\,\ \psi_{(u_1,\dots,u_m)}(\vec{y}) = E^{u_1}\vec{y}.
\end{align*}
Here $P \in \mathrm{GL}_m(\mathbb{Z}/2\mathbb{Z})$ is any element of order $p$, and $E\in \mathrm{GL}_n(\mathbb{Z}/p\mathbb{Z})$ is any non-identity diagonal matrix with diagonal entries
\[ (1,\epsilon_2,\dots,\epsilon_n)\mbox{ with }\epsilon_2,\dots,\epsilon_n\in \{-1,1\}.\]
Plainly $E$ has order $2$ so indeed $\psi$ is well-defined.
  
\begin{lem}\label{lem:phipsi}
For any $b\in (\mathbb{Z}/2\mathbb{Z})^m$ and $c\in (\mathbb{Z}/p\mathbb{Z})^n$, we have $\phi_{\psi_b(c)} = \phi_c$.
\end{lem}

\begin{proof}Writing $b = (u_1,\dots,u_m)$ and $c = (v_1,\dots,v_n)$, we have
\[\psi_b(c) - c  =E^{u_1}
  \begin{bmatrix}v_1\\\vdots  \\ v_n\end{bmatrix} - \begin{bmatrix} v_1\\\vdots \\v_n\end{bmatrix}.\]
If $u_1=0$, then this is the zero vector. If $u_1=1$, then this is equal to
\[\begin{bmatrix} 0 \\ w_2\\\vdots \\ w_n\end{bmatrix},\mbox{ where }w_i = \begin{cases}
0 & \mbox{when }\epsilon_i = 1,\\
-2v_i &\mbox{when }\epsilon_i=-1.
\end{cases}\]
%\[\left[\begin{smallmatrix} 0 \\ \vdots \vspace{1mm}\\ 0 \\ - 2v_n \end{smallmatrix}\right]\]
The first entry is $0$ (since $n\geq 2$) so this lies in the kernel of $\phi$. This implies that $\phi_{\psi_b(c) - c}$ is the identity, as desired.
\end{proof}

Lemma \ref{lem:phipsi} implies that the construction in Section \ref{construct sec} applies. Hence, the set $A=(\mathbb{Z}/2\mathbb{Z})^m\times(\mathbb{Z}/p\mathbb{Z})^n$ endowed with the operations
  \begin{align*}
\left(\left[\begin{smallmatrix}u_1\\\vdots\vspace{1mm}\\u_m\end{smallmatrix}\right],\left[\begin{smallmatrix}v_1\\\vdots\vspace{1mm}\\v_n\end{smallmatrix}\right]\right) \cdot \left(\left[\begin{smallmatrix}x_1\\\vdots\vspace{1mm}\\x_m\end{smallmatrix}\right],\left[\begin{smallmatrix}y_1\\\vdots\vspace{1mm}\\y_n\end{smallmatrix}\right]\right)  &
=\left(\left[\begin{smallmatrix}u_1\\\vdots\vspace{1mm}\\u_m\end{smallmatrix}\right]+P^{v_1}\left[\begin{smallmatrix}x_1\\\vdots\vspace{1mm}\\x_m\end{smallmatrix}\right] ,\left[\begin{smallmatrix}v_1+y_1\\\vdots\vspace{1mm} \\v_n+y_n\end{smallmatrix}\right] \right)\\
\left(\left[\begin{smallmatrix}u_1\\\vdots\vspace{1mm}\\u_m\end{smallmatrix}\right],\left[\begin{smallmatrix}v_1\\\vdots\vspace{1mm}\\v_n\end{smallmatrix}\right]\right) \circ \left(\left[\begin{smallmatrix}x_1\\\vdots\vspace{1mm}\\x_m\end{smallmatrix}\right],\left[\begin{smallmatrix}y_1\\\vdots\vspace{1mm}\\y_n\end{smallmatrix}\right]\right)  &
=\left(\left[\begin{smallmatrix}u_1+x_1\\\vdots\vspace{1mm}\\u_m+u_m\end{smallmatrix}\right],\left[\begin{smallmatrix}v_1\\\vdots\vspace{1mm} \\v_n\end{smallmatrix}\right] +E^{u_1}\left[\begin{smallmatrix}y_1\\\vdots\vspace{1mm} \\y_n\end{smallmatrix}\right] \right)
\end{align*}
becomes a skew brace. Moreover, we have
\begin{align*}
 A &= ((\mathbb{Z}/2\mathbb{Z})^m\times 0)\cdot (0\times (\mathbb{Z}/p\mathbb{Z})^n),\\
 A &= ((\mathbb{Z}/2\mathbb{Z})^m\times 0)\circ (0\times (\mathbb{Z}/p\mathbb{Z})^n),\end{align*}
where both of the factors are sub-skew braces of $A$ which are trivial as skew braces. This skew brace $A$ is not meta-trivial in a lot of cases.

\begin{prop}\label{cri}The skew brace $A$ is not meta-trivial whenever there is an element $v\in \mathbb{Z}/p\mathbb{Z}$ such that $P^{v}-I_m$ has a non-zero entry in the first row.
\end{prop}

%Here $I_m$ denotes the $m\times m$ identity matrix.

\begin{proof}Since $p$ is odd, the calculation in Lemma \ref{lem:phipsi} shows that
\begin{align*}
& \{ \psi_b(c) -c : b \in (\mathbb{Z}/2\mathbb{Z})^m ,\, c\in (\mathbb{Z}/p\mathbb{Z})^n\} \\
&\hspace{1.5cm}= \langle \vec{e}_i : i \in \{2,\dots,n\}\mbox{ with }\epsilon_i = -1\rangle,
\end{align*}
where $\vec{e}_i\in (\mathbb{Z}/p\mathbb{Z})^n$ is the vector having $1$ in the $i$th entry and $0$ everywhere else. Let $i\in \{2,\dots,n\}$ be such that $\epsilon_i=-1$, which exists because $E$ is not the identity matrix by choice. Also, let $j\in\{1,\dots,m\}$ be such that the $j$th entry of the first row of $P^v- I_m$ is $1$. The first entry of $(P^v-I_m)b$ is then $1$, where $b \in(\mathbb{Z}/2\mathbb{Z})^m$ is the vector having $1$ in the $j$th entry and $0$ everywhere else. Taking $c =(v,0,\dots,0)$, we then obtain
\[ \psi_{\phi_c(b)-b}(\vec{e}_i) = \psi_{(P^v-I_m)b}(\vec{e}_i) = E\vec{e}_i = -\vec{e_i},\]
which does not equal to $\vec{e}_i$ because $p$ is odd. It now follows from Proposition \ref{prop:criterion} that $A$ is not meta-trivial.
\end{proof}

It is not hard to find elements $P\in \mathrm{GL}_m(\mathbb{Z}/2\mathbb{Z})$ of odd prime order $p$ such that $P$ satisfies the hypothesis of Proposition \ref{cri}. Here are a few examples when we take $v=1$.
\begin{enumerate}[(1)]
\item $m = 2,\,  p=3,\,$ and $P = \left[\begin{smallmatrix}0 & 1\\ 1 & 1\end{smallmatrix}\right]$.
\vspace{2mm}
\item $m=3,\, p=3,\, $ and $P = \left[\begin{smallmatrix}0 &0& 1\\ 1  & 0& 0\\ 0 & 1 & 0\end{smallmatrix}\right]$.
\vspace{2mm}
\item $m=3,\, p=7,\, $ and $P = \left[\begin{smallmatrix}0 &0& 1\\ 1  & 1& 0\\ 0 & 1 & 0\end{smallmatrix}\right]$.
\vspace{2mm}
\item $m=4,\, p=3,\, $ and $P = \left[\begin{smallmatrix}0 &1&0&0\\1&1&0&0\\0&0&0&1\\0&0&1&1\end{smallmatrix}\right]$.
\vspace{2mm}
\item $m=4,\, p=5,\, $ and $P = \left[\begin{smallmatrix}0 &1&  0 &0\\ 0& 0& 1 & 0\\ 0 & 0 & 0 & 1\\ 1 & 1 &1 & 1\end{smallmatrix}\right]$.
\vspace{2mm}
\item $m=4,\, p=7,\, $ and $P = \left[\begin{smallmatrix}0 &0&  0 &1\\ 1  & 0& 1 & 0\\ 0 & 1 & 1 & 0\\ 0 & 0 & 1 & 0\end{smallmatrix}\right]$.
\end{enumerate}
Note that there is no restriction on $n$ and $E$, except that $n\geq 2$ and $E$ is a diagonal matrix of order $2$.
%\section*{Acknowledgments} 

\end{document}